\documentclass{article} 
\usepackage{amssymb,amsmath,color,graphicx,a4}

\def\eins{\mbox{1\hskip-0.24em l}}
\def\T{^{\sf T}}
\def\mT{^{\sf -T}}
\newcommand{\C}{ {\mathbb C} }
\newcommand{\N}{ {\mathbb N} }
\newcommand{\R}{ {\mathbb R} }
\newcommand{\MM}{{\mathbb M}}

\newcommand{\CC}{{\cal C}}

\newcommand{\PP}{{\cal P}}
\newcommand{\QQ}{{\cal Q}}
\newcommand{\Rc}{{\cal R}}
\newcommand{\Sc}{{\cal S}}

\newcommand{\mynegspace}{\hspace{-0.12em}}
\newcommand{\ltnorm}{\rvert\mynegspace\rvert\mynegspace\rvert}
\newcommand{\rtnorm}{\rvert\mynegspace\rvert\mynegspace\rvert}
\newcommand{\diag}{\,\mbox{diag}}
\newcommand{\bldiag}{\,\mbox{blockdiag}}
\newcommand{\imag}{{\bf i}}

\newcommand{\cc}{{\bf c}}
\newcommand{\gdw}{\,\iff\,}

\newcommand{\weg}[1]{}
\newcommand{\qed}{\qquad\mbox{$\square$}}

\newtheorem{remark}{Remark}[section]
\newtheorem{theorem}{Theorem}[section]
\newtheorem{lemma}{Lemma}[section]

\begin{document}
\title{Stiff Norm Estimates\\ for Variable-Stepsize Peer Two-Step Methods}
\author{Jens Lang \\
{\small \it Technical University Darmstadt,
Department of Mathematics} \\
{\small \it Dolivostra{\ss}e 15, 64293 Darmstadt, Germany}\\
{\small lang@mathematik.tu-darmstadt.de} \\ \\
Bernhard A. Schmitt \\
{\small \it Philipps-Universit\"at Marburg,
Department of Mathematics and Computer Science,}\\
{\small \it Hans-Meerwein-Stra{\ss}e 6, 35043 Marburg, Germany} \\
{\small schmitt@mathematik.uni-marburg.de}}
\maketitle
\begin{abstract}
Peer two-step methods are attractive for the numerical solution of stiff initial value problems because they combine favorable features of Runge–Kutta and multi-step methods, including high-order accuracy, strong stability properties, and the avoidance of order reduction. This paper develops rigorous norm estimates for the propagation operators of variable-stepsize diagonally-implicit $L(\alpha)$-stable Peer methods applied to stiff semi-linear systems. The analysis relies on a class of Peer methods whose propagation matrices possess a grid-independent left eigenvector, enabling the construction of suitable matrix norms for entire coefficient families. We show that this class contains $L(\alpha)$-stable methods of stage order $q=s$ and super-convergence of order $s+1$ with $s$ stages, exceeding the classical stage-order barrier $q\le 2$ of irreducible $s$-stage diagonally-implicit Runge-Kutta methods, and construct such methods up to order 5. Achieving high stage order is one of the principal advantages of Peer methods, especially for stiff differential equations where it helps mitigate order reduction. In addition, Runge-Kutta-type starting procedures of matching local order and efficient diagonally implicit iteration schemes are constructed. Unlike previous approaches based on a posteriori computation of joint spectral radii, the proposed framework exploits the large parameter space of Peer two-step methods to derive stability estimates a priori. The resulting theory provides uniform stability bounds for stiff semi-linear problems on variable grids and establishes a foundation for further applications to optimal control and related boundary value problems. Comparative numerical results are presented for three classical stiff benchmark problems, including three novel Peer two-step methods of order $3$, $4$, $5$, and 
four singly diagonally implicit Runge-Kutta methods with a first explicit step of stage order $2$ and order $3$, $4$, $5$, and $6$.
\end{abstract}

\noindent{\em Key words.} Implicit Peer two-step methods, stiff norm estimates for matrix families, variable stepsizes
\section{Introduction}
Peer methods have been introduced in \cite{SchmittWeiner2004} in linearly-implicit form.
In work by different authors during the last 20 years, they have shown their efficiency also as explicit, implicit-explicit or time-parallel methods, e.g. \cite{ContePaganoPaternoster2023, GerischLangPodhaiskyWeiner2009, KulikovWeiner2015, MontijanoPodhaiskyRandezCalvo2019, PodhaiskyWeinerSchmitt2005, SchmittWeinerBeck2013, SchmittWeinerErdmann2005, SchmittWeinerPodhaisky2005, SchneiderLangWeiner2021, SchroederGerischLang2017, SharifiAbdiHojjati2024}.
Recently, we have applied Peer two-step methods also to optimal control problems \cite{LangSchmitt2026a}, where additional adjoint order conditions have to be satisfied.
In stiff initial value problems, Peer two-step methods are especially attractive combining favorable properties of Runge-Kutta and multi-step methods.
They may have good stability properties including A- and L-stability \cite{MontijanoPodhaiskyRandezCalvo2019,SchmittWeiner2017}, and avoid order reduction for very stiff problems.
However, there appears to be a lack of algebraic tools for rigorously proving norm-contractive stability estimates for $A(\alpha)$- or $L(\alpha)$-stable Peer methods on general grids in stiff problems.
In the present paper, we will provide norm estimates for the propagation operator of $L(\alpha)$-stable Peer methods in time which may be applied to stiff semi-linear problems of the form
\begin{align}\label{semilin}
 y'(t)=f\big(y(t)\big),\,y(0)=y_0, \quad f(y):=Jy+g\big(y\big),
\end{align}
where $J=X\Lambda X^{-1}\in\R^{m\times m}$ is a diagonalizable matrix with eigenvalues in the left complex half-plane and $g:\,\R^n\to\R^n$ is a function with moderate Lipschitz constant $L_g$.
The exact solution of \eqref{semilin} is denoted as $y^\star$.
On variable-stepsize grids, due to their two-step form, the Peer coefficients depend on stepsize ratios and it is demanding to construct one fixed norm applicable to whole matrix families.
Here we need a severe restriction by assuming that a certain left eigenvector of the propagation matrix is grid-independent.
This may seem to be a severe restriction, but a very attractive class of Peer methods with this property has been proposed recently in the context of control problems in \cite{LangSchmitt2025a}.
Methods from this class have not been discussed so far for pure initial value problems and we will show that there  exist $L(\alpha)$-stable methods of even higher order $s+1$ if $s$ is the number of stages.
For all Peer methods, we also construct Runge-Kutta-type starting methods having the appropriate {\em local} order $s+1$.
Since such methods have full coefficient matrices, we also provide fast iteration methods in diagonally-implicit form in order to make the presentation self-contained.
\par
Rigorous stability estimates for multi-step or multi-value methods in variable-stepsize mode are rare.
Here, stability requires that long products $B_n\cdots B_k$, $k<n$, of matrices are uniformly bounded.
This means that the {\em joint spectral radius} of the matrix family $\{B_k\}$ is bounded by one.
However, computation of this radius seems to be an unsolved problem for general matrix families \cite{BlondelNesterov2010}. Guglielmi and Zennaro attacked this problem in \cite{GuglielmiZennaro2001}
for the zero-stability of the multi-step method BDF3 by computing the joint spectral radius and an appropriate polytope norm.
This approach was also pursued in \cite[Section 6]{JackiewiczPodhaiskyWeiner2004} for a two-stage two-step W-method (not a Peer method).
However, we do not follow this a-posteriori approach for existing methods.
Instead we exploit the large number of parameters in Peer two-step methods by combining an a-priori construction of weight matrices for usual matrix norms with the design of the Peer methods itself.
This construction will not be restricted to zero-stability but will apply to stiff semi-linear problems \eqref{semilin} and to corresponding boundary value problems \eqref{AWPyp}-\eqref{ADglyp} arising in ODE constrained optimal control problems.
\par
Peer two-step methods for the numerical solution of \eqref{semilin} have the general redundant form
\begin{align}\label{Peerva}
A_nY_n=&\;B_n Y_{n-1}+h_n K_n F(Y_n),\ n\ge 1,
\end{align}
which is equivalent to
\begin{align}\label{Peerivp}
Y_n=&\;\bar B_nY_{n-1}+h_n\bar K_nF(Y_{n}),\ n\ge 1,
\end{align}
with $\bar B_n=A_n^{-1}B_n$ and $\bar K_n=A_n^{-1}K_n$.
Here $A_n\in\R^{s\times s}$ is a 
nonsingular lower triangular matrix and $K_n\in\R^{s\times s}$ is a diagonal matrix 
with strictly positive entries in order to restrict the size
of the stage equations to $m$ unknowns.
For a shorter notation, a symbol 
like $A_n$ is also understood as its Kronecker product $A_n\otimes I_m$.
Further, $\{t_0,t_1,\ldots,t_{N+1}\}\subseteq[0,T]$ is a grid with stepsizes $h_n=t_{n+1}-t_n$, $n=0,\ldots,N$.
The stage solutions $Y_n=(Y_{ni})_{i=1}^s\in\R^{sm}$ are approximations
of $(y(t_n+c_ih_n))_{i=1}^s$, where all $s$ stages have equal accuracy.
Also, $F(Y_n)=(f(Y_{ni}))_{i=1}^s$ is used. When variable stepsize grids are used, coefficient matrices have to depend on the stepsize ratios $\sigma_n=h_n/h_{n-1},\,n=1,\ldots,N$, which is indicated by the index $n$ in \eqref{Peerva}, \eqref{Peerivp}.
\par
In the context of adjoint Peer methods for optimal control problems, where $A_n$ will be lower triangular and $K_n$ a diagonal matrix, the redundant form \eqref{Peerva} provides more degrees of freedom for the adjoint ODE.
However, for pure initial value problems, the form (\ref{Peerivp}) is often used and in this case $\bar K_n$ will also be lower triangular.
For the stability analysis, we will work with the simplified version \eqref{Peerivp} for pure initial problems.
\par
The paper is organized as follows: In Section \ref{SNorm}, a priori norm estimates for the stability matrix are derived. Three novel Peer two-step methods together with starting methods are constructed in Section \ref{SIPMeth} and the details of the corresponding stability proofs are presented. In Sections \ref{GerrIVP} and \ref{SAPOC}, the global error for pure initial value problems and boundary value problems in optimal control are considered, respectively. Results of numerical experiments are
presented and discussed in Section \ref{NumExp}. The paper concludes with a conclusion in Section \ref{Concl}.
Coefficient sets of novel Peer two-step methods are given in an Appendix A.
\section{Norm estimates for the stability matrix}\label{SNorm}
For a diagonalizable matrix $J$, the stiff part $y'=Jy$ of \eqref{semilin} may be separated into distinct Dahlquist test equations $y'=\lambda y$.
For test equations, the time step \eqref{Peerivp} of the Peer method reduces to $Y_n=\Rc_n(z_n)Y_{n-1}$ with the stability matrix
\begin{align}\label{Stfnz}
 \Rc_n(z)=&\;(I-z\bar K_n)^{-1}\bar B_n,\ z=z_n:=h_n\lambda,\ n\ge1.
\end{align}
A well-established stability concept associated with the test equation is $A(\alpha)$-stability on uniform grids for multi-step type methods.
In this case, the coefficient matrices will not depend on the index $n$ and $A(\alpha)$-stability is usually defined as
\begin{align}\label{Aalpha}
 \varrho\big(\Rc(z)\big)<1\mbox{ for }z\in\C,\ |\arg(-z)|<\alpha,
\end{align}
where $\varrho(\cdot)$ denotes the spectral radius.
Since $\bar K_n$ is non-singular by assumption, we have $\Rc(\infty)=0$ in \eqref{Stfnz} and such methods are also $L(\alpha)$-stable.
However, the spectral radius is not a norm and useless for products of different matrices.
For a rigorous analysis of the global error in the context of stiff initial value problems, estimates for some norm $\|\Rc_n(z)\|$ are required with $z$ in an unbounded part of the left complex half-plane.
Since this norm has to be the same in all time steps, it must not depend on the index $n$.
Also, for $L(\alpha)$-stable methods, bounds may be preferable which show error damping for large eigenvalues as in $\|\Rc(z)\|\to0\,(z\to\infty)$.
Since the form of these norm estimates may be influenced by certain transformations of the coefficient matrices which will be discussed later on, we rather consider the prototype matrix function
\begin{align*}
 \Rc(z)=&\;(I-z\Gamma)^{-1}M,\ z\in\C.
\end{align*}
In the original time step \eqref{Peerivp}, we have $\Gamma=\bar K_n=A_n^{-1}K_n$ and $M=\bar B_n=A_n^{-1}B_n$.
However, additional transformations may be applied to $\bar K_n,\bar B_n$ in order to obtain precise bounds.
Pre-consistency of the scheme requires that $A\eins=B\eins$ or equivalently $M\phi=\phi$ with some transformed eigenvector $\phi$.
Also the scheme \eqref{Peerivp} has to be zero-stable and one may assume that there exists some multiplicative norm such that $\|M\|=1$.
In this case, a simple sufficient estimate for unbounded $z$ is well known by use of the {\em logarithmic norm} $\mu(\Gamma):=\lim_{h\to0}(\|I+h\Gamma\|-1)/h$ subordinate to a given norm \cite{Bauer1962}.
\begin{lemma}
Let $\|M\|=1$, $z\in\C$, and $\mu(z \Gamma)<1$.
Then,
\begin{align}\label{ProdBd}
 \|\Rc(z)\|\le\frac1{1-\mu(z\Gamma)}.
\end{align}
\end{lemma}
{\bf Proof} $\|\Rc(z)\|\le\|M\|\|(I-z\Gamma)^{-1}\|\le 1/(1-\mu(z\Gamma))$ \cite{Bauer1962}.
\qed\\
\par\noindent
For $z\in\R$, $z<0$, we may set $\xi=-z>0$ and since $\mu(-\xi\Gamma)=\xi\mu(-\Gamma)$, the estimate \eqref{ProdBd} shows the desired behavior $\|\Rc(z)\|\to0\,(z\to-\infty)$ for $\mu(-\Gamma)<0$.
Unfortunately, the assumption $\mu(-\Gamma)<0$ is rarely satisfied for the coefficients of practical Peer methods and we have to investigate the interplay between the matrices $M$ and $\Gamma$ in $\Rc(z)$ in more detail.
Since the spectral matrix norm offers more analytic techniques than other norms, we will use this matrix norm $\|\cdot\|_2$ from now on.
For this norm, the assumption $\mu(-\Gamma)<0$ means that the symmetric part of $\Gamma$ is positive definite, which is rarely the case for methods of interest.
The following reformulation, similar to \cite{Schmitt1983}, will be the basis of our approach using the notation $y^\ast=\bar y\T$ for complex vectors.
\begin{lemma}\label{Lsubs}
Let $M,\Gamma\in\R^{s\times s}$, $z\in\C$, and $I-z\Gamma$ be nonsingular.
Then
\begin{align}\label{MNen}
 \|\Rc(z)\|_2^2=&\;\max_{y\not=0}\frac{\|M\T y\|_2^2}{\|y\|_2^2-2\Re(z y^\ast \Gamma y)+|z|^2\|\Gamma\T y\|_2^2},
\end{align}
for all $z\in\C,\,y\in\C^s$, with non-vanishing denominator.
\end{lemma}
{\bf Proof}
Since $\|\Rc(z)\|_2=\|\Rc(z)^\ast\|_2$ and setting $x=(I-\bar z\Gamma\T) y$, we get that
\begin{align*}
 \|\Rc(z)^\ast\|^2=&\;\max_{x\not=0}\frac{\|M\T(I-\bar z\Gamma\T)^{-1}x\|_2^2}{\|x\|_2^2}
 =\max_{y\not=0}\frac{\|M\T y\|_2^2}{\|(I-\bar z\Gamma\T)y\|_2^2}.
\end{align*}
The assertion uses a different representation of the denominator. 
\qed\\
\par\noindent
The quadratic form $y^\ast \Gamma y$ in \eqref{MNen} gives rise to the discussion of the {\em numerical range} or {\em field of values} of the matrix $\Gamma\in\C^{s\times s}$,
\begin{align*}
 {\cal F}_\Gamma:=\{ y^* \Gamma y:\, y\in\C^s,\,\|y\|_2=1\}\subset\C.
\end{align*}
Some well-known properties of this set are \cite{Bauer1962}:\\
\hspace*{1ex}\begin{tabular}{ll}
 a)& ${\cal F}_\Gamma$ is a compact convex set containing all eigenvalues of $\Gamma$.\\
 b)& If $\Gamma$ is a normal matrix, ${\cal F}_\Gamma$ is the convex hull of its eigenvalues.\\
 c)& For real $\Gamma$, it holds that $\mu_2(\Gamma)=\max\{\Re z:\,z\in{\cal F}_\Gamma\}=\lambda_{\max}(\Gamma_S)$,
 where $\Gamma_S:=\frac12(\Gamma+\Gamma\T)$\\& is the symmetric part of $\Gamma$.
\end{tabular}\\{}
\par\noindent
In order to start the discussion, we consider the case of real $z=-\xi\le0$ first having in mind a bound of the form
\begin{align}\label{ziels}
 \frac{\|M\T y\|_2^2}{\|y\|_2^2+2\xi y\T \Gamma_S y+\xi^2\|\Gamma\T y\|_2^2}
 \le\frac1{(1+\xi\nu)^2},\quad\mbox{for all } y\in\R^s,\,y\not=0,
\end{align}
with $\nu>0$.
If $\Gamma$ is nonsingular, the left-hand side of \eqref{ziels} approaches zero for $|\xi|\to\infty$ and the inequality leads to the restriction
\begin{align}\label{nuinfty}
 \nu\le\nu_{\infty}:=1/\|\Gamma^{-1}M\|_2.
\end{align}
In our context, the matrices $M,\Gamma$ will be transformed versions and similar to the coefficient matrices $\bar B_n\sim M$ and $\bar K_n\sim \Gamma$, and both will be related by order conditions.
We always will assume that $\Gamma$ is nonsingular, but $B$ may sometimes be not.
Since $\bar B_n\eins=\eins$ due to pre-consistency, the stability matrix $M$ possesses a left eigenvector $v\T=v\T M$.
Considering the estimate \eqref{ziels}, two different situations depending on the choice of $y$ may come up:
\begin{enumerate}
\item
For $y\cong v$ the norm starts with $\|M\T y\|\cong1$ at $\xi=0$ and $y\T \Gamma y>0$ is required in order to stay below 1 for $\xi>0$.
\item
If, however, the norm on the left starts with $\|M\T y\|\ll1$ at $\xi=0$,  it may grow for $\xi>0$ to some extent allowing for negative values $y\T \Gamma y<0$ before the quadratic term $\xi^2\|\Gamma\T y\|_2^2>0$ in the denominator has a damping effect again.
\end{enumerate}
Since the situation near the dominant eigenvalue one of $\Rc_n(z)$ at $z=0$ is the most critical one, in the following discussion we will construct a fixed similarity transformation for all $\bar B_n$ which separates this eigenvalue for all $\Rc_n(0)$ from a southeast block $M_{se}\in\R^{(s-1)\times(s-1)}$ such that the transformed matrix $M$ has the form
\begin{align}\label{Mblk}
 M=\begin{pmatrix}1&0\T\\0&M_{se} \end{pmatrix},\quad \|M_{se}\|_2<1.
\end{align}
Since this transformation has to be applied to long products
\begin{align}\label{SFProd}
 \Rc_n(z_n)\Rc_{n-1}(z_{n-1})\cdots\Rc_k(z_k),\quad n>k,
\end{align}
it is essential that this transformation is the same for all factors \eqref{Stfnz} and does not depend on the grid index $n$.
This requires that the dominant left eigenvector of $\bar B_n$ is independent of the stepsize ratio which is a severe restriction on the Peer method.
This requirement is fulfilled for adjoint Peer methods due to adjoint order conditions but in general not for Peer methods for initial value problems.
Therefore we will also discuss the construction of such methods in Section~\ref{SIPMeth} below.
After transformation to the form \eqref{Mblk}, both the left and right eigenvectors of $M$ are equal to $e_1$ and in view of the case number~1 above we will need to show that $e_1\T \Gamma e_1>0$.
\par
Before proceeding with our main discussion, we will shortly describe an alternative approach which may also be of interest.
It will be restricted to  proving the simple estimate
\begin{align}\label{Norm1}
 \|\Rc(-\xi)\T\|_2=\|M\T(I+\xi \Gamma)\mT\|_2\le 1,\quad z=-\xi\in\R,\,\xi\ge0,
\end{align}
for nonsingular $\Gamma$.
\begin{remark}
By Lemma~\ref{Lsubs} condition \eqref{Norm1} is equivalent with requiring that
\begin{align}\label{DefBed0}
 0\le&\;y\T(I-MM\T)y+2\xi y\T \Gamma_Sy+\xi^2 y\T \Gamma\Gamma\T y\mbox{ for all }\xi\ge0,\,y\in\R^s,
\end{align}
with the symmetric part $\Gamma_S:=\frac12(\Gamma+\Gamma\T)$ of $\Gamma$.
We observe that for any fixed $y\in\R^n$ the right hand side is a quadratic polynomial in $\xi$.
For $\|M\|_2=1$, condition \eqref{DefBed0} is trivially satisfied if $y\T \Gamma_Sy\ge0$.
In the opposite case $y\T \Gamma_Sy<0$, the minimum value of the right-hand side is computed easily and non-negativity of the minimum amounts to 
\begin{align*}
 y\T(I-MM\T)y\ge \frac{(y\T \Gamma_Sy)^2}{y\T \Gamma\Gamma\T y}.
\end{align*}
Setting $x:=\Gamma\T y$, we reformulate this condition as
\begin{align*}
  \|x\|_2^2x\T \Gamma^{-1}(I-MM\T)\Gamma\mT x\ge (x\T \Gamma^{-1} \Gamma_S\Gamma\mT x)^2,\mbox{ when }x\T \Gamma^{-1} \Gamma_S\Gamma\mT x<0.
\end{align*}
Normalizing $\|x\|_2=1$, taking square roots and reminding the considered case with the negative sign of $x\T \Gamma^{-1}\Gamma_S\Gamma\mT x=x\T\big((\Gamma^{-1})_S\big)x$, we obtain the following sufficient condition for \eqref{Norm1}, where the variable $\xi$ has been eliminated:
\eqref{Norm1} holds if
\begin{align*}
  x\T \Gamma^{-1} \Gamma_S\Gamma\mT x+\sqrt{x\T \Gamma^{-1}(I-MM\T)\Gamma\mT x}\ge 0\mbox{ for all }x\in\R^s,\,\|x\|_2=1.
\end{align*}
Unfortunately, the left hand side is not a convex map in $x$ and we see no obvious way to make mathematically rigorous decisions here. 
\end{remark}
Considering generalizations of the estimate \eqref{ziels} for complex values of $z$, a choice for the form of the upper bound is required.
The behavior of the exponential functions suggests the use of $1/(1-\nu\Re z)$.
However, we would like to apply the maximum principle in proving the bound and rather consider estimates of the form
\begin{align}\label{NDaempf}
 \|\Rc(z)\|_2=\|(I-z\Gamma)^{-1}M\|_2\le\frac1{|1-\nu z|},\quad \nu\ge0,
\end{align}
with $z$ from an appropriate part of the complex plane containing the negative real axis, at least.
Obviously, \eqref{NDaempf} is equivalent with the requirement that
\begin{align*}
 |(1-\nu z)y^\ast (I-z\Gamma)^{-1}Mx|\le1\quad\forall x,y\in\C^s,\,\|x\|_2=\|y\|_2=1,
\end{align*}
where each scalar map $z\mapsto (1-\nu z)y^\ast (I-z\Gamma)^{-1}Mx$ is holomorphic in the left complex halfplane since $\Gamma\sim\bar K$ has positive real eigenvalues.
We now present a simple criterion for \eqref{NDaempf} to hold.
The definiteness assumption in the following lemma for $\xi>0$ is a bit stronger than needed, but it has the advantage that it may be verified more easily by the Sylvester criterion.
\begin{lemma}\label{LOmdefin}
The estimate \eqref{NDaempf} holds on some complex ray $z=-\xi(1+\imag\eta)$, $\xi\ge0$, with fixed $\eta\in\R$, where $I-z\Gamma$ is nonsingular, if  $\|M\|_2\le1$ and if the Hermitean matrix
\begin{align}\label{Omdef}
 \Omega(\xi):=I-MM\T+2\xi(\Gamma_S+\imag\eta \Gamma_A-\nu MM\T) +\xi^2(1+\eta^2)\big(\Gamma\Gamma\T-\nu^2MM\T\big)
\end{align}
is positive definite for all $\xi>0$.
Here, $\Gamma_S=\frac12(\Gamma+\Gamma\T)$ and $\Gamma_A=\frac12(\Gamma-\Gamma\T)$ are the symmetric and anti-symmetric parts of $\Gamma$, respectively.
A consequence of this assumption is that $\nu<\nu_\infty$, see \eqref{nuinfty}.
\end{lemma}
{\bf Proof}
With the reformulation \eqref{MNen}, the estimate \eqref{NDaempf} is equivalent with
\begin{align*}
 0\le y^*\big(I-(z\Gamma+\bar z\Gamma\T)+|z|^2\Gamma\Gamma\T\big)y-|1-\nu\ z|^2\|M\T y\|_2^2
\end{align*}
for all $y\in\C^s$.
Since $|z|^2=\xi^2(1+\eta^2)$, reordering terms leads to \eqref{Omdef}.
Definiteness for $\xi\to\infty$ requires the stated restriction on $\nu$.
\qed\\
\par\noindent
Since with fixed vectors $x,y\in\C^s$ on the unit sphere, the map $z\to(1-\nu z)y^\ast{\cal R}(z)x$ is holomorphic, by the maximum principle  verification of the lemma implies that the norm estimate holds in the full sector $\{z\in\C:\,|\Im z|\le-\eta\Re z\}$.
\par
In contrast to the characterization discussed in Remark~1, condition \eqref{Omdef} may be verified rigorously despite its dependence on a real and unbounded parameter $\xi\ge0$.
Firstly, the assumption $\|M\|_2=1$ is required by zero-stability and shows that $\Omega(0)=I-MM\T\succeq0$.
Each entry of $\Omega(\xi)$ is a quadratic polynomial in $\xi$.
Then, definiteness of $\Omega(\xi),\,\xi>0,$ may be shown by computing its formal LU-decomposition $\Omega=LU$ and showing that all diagonal elements of $U$ are positive.
This follows from the Sylvester Law of Inertia since these elements of $U$ coincide with that of the matrix $L^{-1}\Omega(L^{-1})^*$ which is congruent with $\Omega$ and a diagonal matrix.
If $M$ is in block-diagonal form \eqref{Mblk}, the first pivot step in the LU decomposition can be characterized more precisely.
Since $Me_1=M\T e_1=e_1$, we have
\begin{align}\label{Omg1}
\Omega(\xi)e_1=&\;2\xi(\Gamma_S+\imag\eta \Gamma_A-\nu I)e_1+\xi^2(1+\eta^2)
\big(\Gamma\Gamma\T-\nu^2I\big) e_1
\end{align}
and $\omega_{11}(\xi):=e_1\T\Omega(\xi)e_1=2\xi(e_1\T \Gamma e_1-\nu)+\xi^2(1+\eta^2)\big(\|\Gamma\T e_1\|_2^2-\nu^2\big)$.
In the context of Peer methods, the matrices $\Gamma,M$ will be transformed versions of $\bar K,\bar B$ and are related by order conditions.
We will see below in \eqref{blkivp} that $e_1\T \Gamma e_1=1$ and that $\nu\ll1$.
Hence, $e_1\T\Omega(\xi)e_1>0$ for $\xi>0$ and the factor $\xi$ cancels out in the quotients $\Omega(\xi)e_1/e_1\T\Omega(\xi)e_1=O(1),\,\xi\to0,$ of the first elimination step.
During elimination, all elements of $U(\xi)=\big(u_{ij}(\xi)\big)$ will be rational functions of $\xi$ of increasing degree.
In fact, the functions
\begin{align*}
\zeta_i(\xi):=\frac1{\xi}u_{11}(\xi)\cdots u_{ii}(\xi),\ i=1,\ldots,s,
\end{align*}
are polynomials of degree $2i-1$ since $\xi\cdot\zeta_i(\xi)$ is the $i$-th leading principal minor of $\Omega(\xi)$ and by the Sylvester criterion positivity $\zeta_i(\xi)>0,\,i=1,\ldots,s$, proves definiteness $\Omega(\xi)\succ0$ for $\xi>0$.
\par
Now, for polynomials $\zeta(\xi)$ there is a simple criterion for positivity on $\R_+$ which may be effectively verified by inspection of a finite number of coefficients.
These coefficients are rational numbers if all numbers in $\Gamma,M$ are rational.
The criterion 
\begin{align}\label{Polpos}
 \zeta(\xi)>0\;\forall\xi>0,\mbox{ if }
 \left\{\begin{array}{l}
  \mbox{for some $\ell\in\N$ the polynomial $(1+\xi)^\ell\zeta(\xi)$}\\
  \mbox{possesses nontrivial coefficients, all being positive,}
 \end{array}\right.
\end{align}
is sufficient, obviously, but according to Bernstein and P\'olya also necessary \cite{Polya1928}.

\section{Peer two-step methods for initial-value problems}\label{SIPMeth}
The criterion for norm stability will be applied to our recently developed adjoint Peer methods \cite{LangSchmitt2025a} for optimal control.
However, since Peer methods with a constant left dominant eigenvector $v\T=v\T\bar B_n$ for all $\bar B_n$ apparently have not been considered so far in the literature, new methods of this form will be designed.
In fact, we will start with these methods since the simplest of them is the only one where the stability criterion may be checked in some cases manually without the aid of algebra software.
We are focusing on the promising class of LSRK methods ({\em Last Stage is Runge-Kutta}, \cite{LangSchmitt2025a}) which have not been studied before for this problem class.
For local order $r\le s$, the consistency conditions are \cite{LangSchmitt2025a}
\begin{align}\label{Ordvr}
 A_nV_r=B_nV_r\PP_r^{-1}S_n^{-1}+K_nV_r\tilde E_r\gdw
 V_r=\bar B_nV_r\PP_r^{-1}S_n^{-1}+\bar K_nV_r\tilde E_r,
\end{align}
with $\bar B_n=A_n^{-1}B_n$, $\bar K_n=A_n^{-1}K_n$.
The other matrices in \eqref{Ordvr} are the Vandermonde matrix $V_r=\big(\eins,\cc,\ldots,\cc^{r-1}\big)$ with $\cc^j=(c_1^j,\ldots,c_s^j)\T$, the Pascal matrix $\PP_r=\big({j-1\choose i-1}\big)_{i,j=1}^r$, and the scaled shift matrix $\tilde E_r=\big(i\delta_{i+1,j}\big)_{i,j=1}^r$ which commutes with $\PP_r=\exp(\tilde E_r)$.
The dependence on $n$ enters through the stepsize ratio $\sigma_n=h_n/h_{n-1}$ only in $S_n:=S(\sigma_n)$, $S(\sigma):=\diag(1,\sigma,\ldots,\sigma^{r-1})$, where its dimension is easily derived from the context.
For pure initial value problems, order $r\ge s$ is usually used and $A_n\equiv I$ for simplicity.
In this case, the matrix $\bar B_n=B_n$ is given by
\begin{align}\label{Bsig}
\bar B_n=B(\sigma_n)=(V_s-\bar K_nV_s\tilde E_s)S_n\PP_sV_s^{-1}.
\end{align}
Considering similarity transformations not depending on $\sigma$ which transform $\bar B_n$ to block form \eqref{Mblk}, it is obvious that besides the right eigenvector $\eins=\bar B\eins$ also the corresponding left eigenvector $v$ must not depend on $\sigma_n$.
This condition leads to the following result which we formulate with the 'bar'-matrices in order to make it also applicable to the adjoint Peer methods below.
\begin{lemma}\label{LLSRK}
Let the matrix $\bar K_n$ satisfy $\eins\T V_s^{-1}\bar K_n\cc^{j-1}=1/j,$ $j=1,\ldots,s-1$.
Then, $v\T=\eins\T V_s^{-1}$ is the constant left eigenvector $v\T=v\T\bar B_n$ of the matrix \eqref{Bsig}.
If also $c_s=1$, then $v\T=e_s\T$ and the last stage of the Peer method is a Runge-Kutta stage
\begin{align}\label{LSRK}
 Y_{ns}=Y_{n-1,s}+h_n\sum_{j=1}^s\bar\kappa_{sj}F(Y_{nj}).
\end{align}
\end{lemma}
\textbf{Proof}
 The proof simplifies by considering the transformed matrices $\tilde B_n=V_s^{-1}\bar B_nV_s$, 
 $\tilde K_n=V_s^{-1}\bar K_nV_s$.
The assumption on $\bar K_n$ is equivalent with $\eins\T\tilde K_ne_j=1/j,\,j<s,$ which means that $e_1\T\PP_s(I-\tilde K_n\tilde E_s)e_j=0,\,j=2,\ldots, s$.
Regrouping the factors in \eqref{Bsig} in order to make $S_n$ the right-most one, the eigenvector condition becomes
\begin{align*}
 e_1\T\PP_s \tilde B_n\PP_s^{-1}=e_1\T\PP_s(I-\tilde K_n\tilde E_s)S_n= e_1\T\mbox{ for all }\sigma_n\in\R.
\end{align*}
Hence, $v\T=e_1\T \PP_s V_s^{-1}=\eins\T V_s^{-1}$ is the eigenvector of $\bar B_n$ and does not depend on $\sigma_n$.
For $c_s=1$ we have $e_s\T V_s=\eins\T$ leading to $\eins\T V_s^{-1}=e_s\T$.
\qed\\
\par\noindent
The property \eqref{LSRK} was named LSRK {\em (Last Stage is Runge-Kutta)} in \cite{LangSchmitt2025a} since $Y_{n-1,s}\cong y(t_n)$ for $c_s=1$.
It leads to some favorable properties of the Peer method one of which is discussed now.
Denoting vectors of stage values of the exact solution by boldface like ${\bf y}_n=\big(y^\star(t_{ni})\big)_{i=1}^s$, ${\bf y}_n'=\big((y^\star)'(t_{ni})\big)_{i=1}^s$, Taylor expansion at $t_n$ of the truncation error of the Peer step \eqref{Peerivp} gives
\begin{align}\label{lokfer}
 \tau_n:=&\;{\bf y}_n-A_n^{-1}(B_n{\bf y}_{n-1}+h_nK{\bf y}_n')
 =h_n^r \beta_r(\sigma_n)y^{(r)}(t_n)+O(h_n^{r+1}),\\\label{betar}
  \beta_r(\sigma)=&\;\frac1{r!}(\cc^r-\bar B_n(\cc-\eins)^r\sigma^{-r}-r\bar K_n\cc^{r-1}).
\end{align}
The norm of the vector $\beta_r$ is again considered to be the error constant,
\begin{align}\label{errr}
 err_r:=\|\beta_r(1)\|_2.
\end{align}
In the analysis of the global error usually one order is lost.
Since long products \eqref{SFProd} with factors ${\cal R}_j(0)$ of the form \eqref{Mblk} converge to a rank-1-matrix, this loss may be prevented by a {\em super-convergence} effect if the (constant) left eigenvector of $B_n$ is orthogonal to $\beta_r$, canceling the leading error term, which means $e_s\T\beta_{r}=0$ for LSRK methods.
But surprisingly it is also possible to push the local and global order to $r=s+1$.
We combine both results now.
\begin{lemma}\label{LOrdsp1}
Let the Peer method posses local order $s$ with the matrix $\bar B_n$ given in \eqref{Bsig} and let $\varphi(t):=(t-c_1)\cdots(t-c_s)$ be the node polynomial of the method.
Then, the additional $s$ conditions for local order $s+1$ may be written as the vector equation
\begin{align}\label{Ordsp1}
 \varphi(\eins+\sigma_n\cc)=\sigma_n\bar K_n\varphi'(\eins+\sigma_n\cc).
\end{align}
For LSRK methods, the additional condition for super-convergence $e_s\T\beta_{s+1}(\sigma)\equiv0$ of {\em global} order $s+1$ is
\begin{align}\label{Supkv}
 e_s\T \bar K_n\cc^s=\frac1{s+1}.
\end{align}
\end{lemma}
{\bf Proof}
The first condition has been shown before in previous work, but not exactly in the form \eqref{Ordsp1}.
However, combining the lengthy computations from (4.2) in \cite{SchmittWeinerPodhaisky2005} for general $\sigma$ but diagonal $\bar K_n$ and Lemma~3.1 in \cite{SchmittWeinerBeck2013} for triangular $\bar K_n$, the condition is shown easily.
\par 
For LSRK methods, the dominant left eigenvector is $e_s\T=e_s\T \bar B_n$.
Hence, super-convergence of order $s+1$ requires that
\begin{align*}
 0\stackrel!=&\;e_s\T\beta_{s+1}=c_s^{s+1}-\sigma_n^{-s-1}e_s\T\bar B_n(\cc-\eins)^{s+1}-(s+1)e_s\T\bar K_n\cc^s\\
 =&\;1-0-(s+1)e_s\T\bar K_n\cc^s.
 \qed
\end{align*}
\begin{remark}
For LSRK methods the two conditions \eqref{Ordsp1}, \eqref{Supkv} show that the last row of $\bar K_n$ contains the constant weights of the quadrature rule associated with the nodes in $\cc$ and requires that the rule is exact for polynomials of degree $s$.
Hence, the node polynomial has to be orthogonal to constants, i.e., $\int_0^1\varphi(t)dt=0$.
\end{remark}
An important consequence of condition \eqref{Ordsp1} is that even local order $s+1$ may be achieved for general grids if also the coefficient matrix $\bar K_n$ is $\sigma$-dependent.
To our knowledge, this case has not been discussed so far in the literature for $L(\alpha)$-stable Peer methods.
In the mentioned papers \cite{SchmittWeinerPodhaisky2005}, \cite{SchmittWeinerBeck2013} condition \eqref{Ordsp1} was used for obtaining super-convergence of order $s$, while order $s+1$ was obtained previously only by use of additional explicit terms $F(Y_{n-1})$ in the Peer step \cite{SchneiderLangWeiner2021}. 
The additional order condition \eqref{Ordsp1} is easily satisfied in the last $s-1$ components by solving it for the elements $\kappa_{i1}$, $i>1$, which are not restricted in any way.
However, due to the triangular form of $\bar K_n$ in the first component, the condition requires that
\begin{align}\label{kappa11}
\frac1{\bar\kappa_{11}^{[n]}}=\frac{\sigma_n\varphi'(1+\sigma_nc_1)}{\varphi(1+\sigma_nc_1)}
 =\big(\ln\varphi(1+\sigma_n t)\big)'\big|_{c_1},
\end{align} 
which needs to be positive.
Indeed, for LSRK methods with nodes in $(0,1]=(0,c_s]$, it is obvious that $\bar\kappa_{11}^{[n]}>0$ is a smooth bounded function since $\varphi(t)$ is an increasing positive function for $t>1$ and the factor $\sigma_n$ is canceled by $\varphi(1+\sigma t)=\varphi(c_s+\sigma t)=O(\sigma t),\,t\to0$.
\par
Finally, we have to note that super-convergence is a consequence of the fact that long products $\bar B_n\cdots\bar B_k$, $n>k$, converge for $n-k\to\infty$ to a rank-one matrix which is $\bar B^\infty=\eins e_s\T$ for LSRK methods, see \cite{LangSchmitt2025a}.
However, order $s+1$ will be observed in practice only if the convergence $\bar B_n\cdots\bar B_k\beta_{s+1}\to0$ is fast enough.
With the decompositions \eqref{Mblk} resp. \eqref{blkivp} below, this requirement may be formulated as
\begin{align}\label{kontrkt}
 \varrho(M_{se})=\varrho(B_{se})\le\gamma<1
\end{align}
where $\gamma=0.8$ works well in practice, \cite{LangSchmitt2022b}.
Consult \eqref{blkivp} for the definition of $B_{se}$.
\par
In \cite{LangSchmitt2025a} a decoupling transformation \eqref{Mblk} was constructed for general adjoint Peer methods.
However, for LSRK methods we may present a simple transformation $W_b$ to block-diagonal form for all $\bar B_n$, which will be a basic factor of the full weight matrix $W$.
\begin{lemma}\label{LWb}
The matrix
\begin{align*}
 W_b:=\begin{pmatrix}\eins_{s-1}&I_{s-1}\\[1mm]
 1&0_{s-1}\T \end{pmatrix}
\end{align*}
satisfies $e_1\T W_b^{-1}=e_s\T$, $W_be_1=\eins$.
Then, for any matrix $\bar B_n,\bar K_n$, satisfying \eqref{Bsig} and the assumptions of Lemma~\ref{LLSRK} with $c_s=1$, it holds that
\begin{align}\label{blkivp}
 W_b^{-1}\bar B_nW_b=\begin{pmatrix}1&0_{s-1}\T\\[1mm]
 0_{s-1}&B_{se}(\sigma_n) \end{pmatrix},\quad
 e_1\T(W_b^{-1}\bar K_nW_b)e_1=1.
\end{align}
\end{lemma}
{\bf Proof}
A simple computation shows 
$W_b^{-1}\bar B_nW_be_1=W_b^{-1}\bar B_n\eins=W_b^{-1}\eins=e_1$ 
and $e_1\T W_b^{-1}\bar B_nW_b=e_s\T\bar B_nW_b=e_s\T W_b=e_1\T$.
Finally, $e_1\T W_b^{-1}\bar K_nW_be_1=\eins\T\tilde K_n e_1=1$ by the assumptions of Lemma~\ref{LLSRK} on $\bar K(\sigma_n)$.
\qed\\
\par\noindent
According to the lemma, for vectors $y\cong e_1$ the left-hand side in \eqref{ziels} is indeed a monotonically decreasing function of $\xi>0$ and a necessary condition for \eqref{ziels} to hold is satisfied here.
Also, for $\nu<1$ the first pivot $\omega_{11}(\xi)=\zeta_1(\xi)$ in the LU decomposition of $\Omega(\xi)$ is always positive for $\xi>0$ since $\omega_{11}(\xi)\ge 2\xi(g_{11}-\nu)+\xi^2(1+\eta^2)(g_{11}^2-\nu^2)$ according to \eqref{Omg1} and this first element need not be discussed below for each method.
And the element $g_{11}=e_1\T \Gamma e_1=\eins\T\bar K_ne_1=1$ in \eqref{blkivp} will not be changed by additional scaling transformation discussed now.
\par
We note that for variable stepsizes the south-east block $B_{se}$ in \eqref{blkivp} depends on the stepsize ratio $\sigma_n$.
In order to minimize its norm for specific Peer methods, an additional constant similarity transformation of the same block-form $1\oplus R_{se}$ will be applied leading to a final weight matrix $W=W_b(1\oplus R_{se})\T$.
Under such transformations the properties in \eqref{blkivp} remain unchanged.
Due to the unitary invariance of the spectral norm, $R_{se}$ can be restricted to (upper) triangular form since $R$ may be considered to be the triangular factor of the QR-decomposition $QR=X$ of a general factor in $W_bX\T$.
Use of the similarity transformation with the fixed matrix $W$ in the spectral norm corresponds to estimates in the weighted norm defined by
\begin{align}\label{Wnorm}
 \ltnorm \Theta\rtnorm :=\|W^{-1}\Theta W\|_2,\ \Theta\in\C^{s\times s},
\end{align}
which is induced by the vector norm $\ltnorm x\rtnorm:=\|W^{-1}x\|_2,\,x\in\C^s$.
Since the verification procedure of the norm estimate consists of three different steps, we summarize it now for LSRK Peer methods.
For practical reasons we use a set ${\cal S}\subseteq(0,2)$ of admissible stepsize ratios which may be an interval ${\cal S}=[\underline\sigma,\bar\sigma]$ or a finite set like ${\cal S}=\{\underline\sigma,1,\bar\sigma\}$ with $\underline\sigma<1<\bar\sigma$.
Since we now are talking about matrix families, the limit of the admissible damping factors has to be redefined as
\begin{align}\label{nuinfneu}
 \nu_\infty:=1/\max\{\ltnorm \bar K(\sigma)^{-1}\bar B(\sigma)\rtnorm:\,\sigma\in{\cal S}\}.
\end{align}
Obviously, the possible choices of the set ${\cal S}$, the aperture $\eta$ of the complex sector, and the damping factor $\nu$ are related and a careful selection is required for successful applications.
\begin{theorem}\label{TNStab}
Let $\bar B(\sigma),\bar K(\sigma)$ be the coefficients of an LSRK Peer method satisfying the order conditions \eqref{Ordvr} with $r\ge 2$ and let $W\in\R^{s\times s}$ be such that $M=W^{-1}\bar B(\sigma)W$ is in block form \eqref{Mblk} with $\|M_{se}\|_2<1$ for $\sigma\in{\cal S}$.
Choose $\nu\ge0$ such that $\nu<\min\{1,\nu_\infty\}$, see \eqref{nuinfneu}, and $\eta>0$.
Compute the formal LU-decomposition of $\Omega(\xi)=L(\xi)U(\xi)$ from \eqref{Omdef} and let $\zeta_i(\xi)=\xi^{-1}u_{11}(\xi)\cdots u_{ii}(\xi)$, $i=1,\ldots,s$.
\par
If this decomposition $L(\xi)U(\xi)=\Omega(\xi)$ exists, if the polynomials $\zeta_i(\xi)$ are nontrivial for $i=1,\ldots,s$ and there exist $\ell_i\in\N_0$ such that $(1+\xi)^{\ell_i}\zeta_i(\xi)$ possesses no negative coefficients, then $\Omega(\xi)$ is positive definite for $\xi>0$ and it holds that
\begin{align}\label{Nrmsek}
 \ltnorm\big(I-z\bar K(\sigma)\big)^{-1}\bar B(\sigma)\rtnorm\le\frac1{|1-\nu z|},
 \ z\in\C,\ |\Im z|\le-\eta\Re z,\,\sigma\in{\cal S}.
\end{align}
\end{theorem}
{\bf Proof}
Repeating previous arguments, positive coefficients of $(1+\xi)^{\ell_i}\zeta_i(\xi)$ prove that $\zeta_i(\xi)>0$, $i=1,\ldots,s$, for $\xi>0$.
Hence, all principal minors $\xi\,\zeta_i(\xi),\,i=1,\ldots,s$, of the matrix $\Omega(\xi)$ are positive showing its definiteness for $\xi>0$ by the Sylvester criterion.
With $M=W^{-1}\bar B(\sigma)W$, $\Gamma=W^{-1}\bar K(\sigma)W$ Lemma~\ref{LOmdefin} completes the proof.
\qed\\
\par\noindent
In the design of the following Peer methods, the main objective was maximizing the angle $\alpha$ of $L(\alpha)$-stability, not the size of the damping parameter $\nu$, for instance.
Only in a second step scalings $R_{se}$ were sought for with the aim of maximizing $\nu$.
Methods for which no appropriate scaling $R_{se}$ could be found were sorted out.
\par
If ${\cal S}$ is a finite set, the criterion may be checked for all its elements.
Although it might be quite tedious to implement, we also describe shortly, as a proof of concept, a simple procedure for verifying \eqref{Polpos} in the case that the set ${\cal S}$ of stepsize ratios is an interval.
For order $s+1$, the elements $\kappa_{i1}$ of $\bar K(\sigma)$ are rational functions of $\sigma$, all with the same denominator $\varphi'(1+\sigma c_1)>0$ for $c_i\in(0,1],\,i=1,\ldots,s$, see \eqref{kappa11}.
Multiplying the modified polynomials $(1+\xi)^{\ell_i}\zeta_i(\xi)$ by the least common denominator $\varphi'(1+\sigma c_1)^d$, $d\ge1$, we may write 
\begin{align*}
 \hat\zeta_i(\xi)=\varphi'(1+\sigma c_1)^d(1+\xi)^{\ell_i}\zeta_i(\xi)
 =\sum_{k=0}^{\ell_i+2i-1}\hat\zeta_{ik}(\sigma)\xi^k,\quad 1\le i\le s,
\end{align*}
where the coefficients $\hat\zeta_{ik}$ are polynomials in $\sigma$.
With a change of variables $\sigma=1+\delta$ this means that
\[ \hat\zeta_{ik}(\sigma)=\hat\zeta_{ik}(1+\delta)=\hat\zeta_{ik}(1)+\sum_{j=1}\hat\zeta_{ikj}\delta^j\]
with finite sums.
Now, assuming $\hat\zeta_{ik}(1)>0$ and considering only small intervals around one like ${\cal S}\subseteq[1-\delta_0,1+\delta_0]$, $0<\delta_0\ll1$, we may ensure non-negativity of $\zeta_{ik}(\sigma)$ by requiring positivity of a lower bound such as
\begin{align}\label{zetadlt}
 \hat\zeta_{ik}(1+\delta)\ge\hat\zeta_{ik}(1)-|\delta|\sum_{j=1}|\hat\zeta_{ikj}|\delta_0^{j-1}\stackrel!\ge0,\ i=2,\ldots,s.
\end{align}
Here, we remind that $\zeta_1(\xi)>0,\,\xi>0$, according to \eqref{Omg1}.
Then, denoting $\hat\delta=\min\{\delta_0,\delta'\}$, $\delta'=\min_{i=2}^s\min_{k=0}^{\ell_i+2i}\hat\zeta_{ik}(1)/\sum_{j=1}|\hat\zeta_{ikj}|\delta_0^{j-1}$, the criterion \eqref{Polpos} will be satisfied for $\sigma\in{\cal S}=[1-\hat\delta,1+\hat\delta]$.
\par
The assumption of stepsize ratios with $|\sigma_n-1|\le\delta_0$ with small $\delta_0\cong0.1$ below may seem quite restrictive.
However, adaptive final grids for highly accurate solutions of mildly stiff problems often satisfy this property, see \cite{LangSchmitt2026a}.
\par
The construction of appropriate weight matrices $W$ used in Theorem~\ref{TNStab} will be incorporated in the design of the new Peer methods below.
In proving the norm estimate \eqref{Nrmsek} we will always consider two cases.
Since the damping factor $\nu$ will vanish if the aperture $\eta$ approaches its maximal possible value, the first case will discuss nearly extreme apertures with no damping $\nu=0$.
For smaller aperture $\eta$, the damping factor soon will not be limited by bumps near the imaginary axis but by the limiting behavior for $z\to\infty$.
Hence, the second case will be some convenient choice with smaller aperture and $\nu<\nu_{\infty}$.
Since the norm estimate \eqref{Nrmsek} is analytically much stronger than information on the spectral radius as in \eqref{Aalpha}, it should be expected that the aperture $\eta$ of the complex sector is smaller than that of the $A(\alpha)$ sector, i.e. $\eta<\tan(\alpha)$.
\subsection{Iteration for starting methods}
In the context of pure initial value problems, the computation of approximations for the starting step of multi-step-type methods is usually obtained by unnamed standard methods.
However, in order to present a self-contained procedure and in view of the favorable experience in the context of adjoint methods, we consider again Runge-Kutta-type starting methods for the new Peer schemes.
With the required high orders of approximation, a triangular form of $K_0$ or $A_0$ is no longer possible for $s>2$.
In order to avoid coupled systems of larger dimensions than $m\times m$, we again resort to block-Gauss-Seidel iterations.
Since this iteration looks simpler for methods in the form \eqref{Peerva} than for \eqref{Peerivp}, we consider the starting method
\begin{align}\label{Startab}
 A_0Y_0= a\otimes y_0+h_0b\otimes f(y_0)+h_0 F(Y_0)
\end{align}
with $K_0=I$.
Since our new IVP-methods will have high orders, this step also contains the additional evaluation of $y'(t_0)$, see \cite{LangSchmitt2022a}.
Since the time step \eqref{Startab} is applied once only, local order $s+1$ is sufficient and the mentioned paper gives the corresponding order conditions as
\begin{align}\label{OBstrt}
 A_0V_{s+1}=ae_1\T+be_2\T +V_{s+1}\tilde E_{s+1}.
\end{align}
The truncation error for \eqref{Startab} with \eqref{OBSred} and $s\ge 2$ is
\begin{align}\label{lokfe0}
 \tau_0=h_0^{r}\beta_{r}y^{(r)}(t_0)+O(h_0^{r+1}),
 \ \beta_{r}=\frac1{r!}(\cc^{r}-rA_0^{-1}\cc^r),
\end{align}
with $r=s+1$.
The two first columns of \eqref{OBstrt} immediately yield $a=A_0\eins$ und $b=A_0\cc-\eins$ for the coefficient vectors in \eqref{Startab}.
Introducing $D_c=\diag(c_i)$, the remaining columns of \eqref{OBstrt} read
\begin{align}\label{OBSred}
 A_0\big(\cc^2,\ldots,\cc^s\big)=(2\cc,3\cc^2,\ldots)
 \gdw
 A_0D_c^2V_{s-1}=D_c V_{s-1}\diag(2,\ldots,s).
\end{align}
These conditions may be solved with the entries $A_0e_s$ of the last column as parameters or with a slack vector $\phi\in\R^s$ as
\begin{align}\label{A0par}
  A_0=D_c(V_sD_+V_s^{-1})D_c^{-2}-\phi e_s\T(D_c^2 V_s)^{-1},
\end{align}
where $D_{+}=\diag(2,\ldots,s+1)$.
\par
Considering an approximation $\tilde A_0$ to $A_0$, one iteration for the Newton step for \eqref{Startab} reads
\begin{align}\label{Nwtit}
 \big(\tilde A_0-h_0 \textbf{J}_0\big)(Y_0^{(k+1)}-Y_0^{(k)})
 =a\otimes y_0+h_0b\otimes f_0+h_0F(Y_0^{(k)})-A_0Y_0^{(k)},
\end{align}
$k\ge 0$, with the Jacobian $\textbf{J}_0=\bldiag_i(J_{0i})$, $J_{0i}=\nabla_y f(y_0),\ i=1,\ldots,s$.
Now, if $\tilde A_0$ is lower triangular and all sub-diagonals are the same as in $A_0$, then, the difference $R_0:=A_0-\tilde A_0$ is upper triangular and the step can be  performed without additional memory requirements.
In fact, \eqref{Nwtit} may be solved successively as in \cite{LangSchmitt2026a},
\begin{align}\notag
 &\big(\tilde a_{ii}^{(0)}I-h_0J_{0i}\big)(Y_{0i}^{(k+1)}-Y_{0i}^{(k)})\\\notag
 &=a_iy_0+h_0b_if_0+h_0f(Y_{0i})-\sum_{j=1}^{i-1}\tilde a_{ij}^{(0)}Y_{0j}^{(k+1)}-\sum_{j=i}^{s}
 \left( r_{ij}^{(0)} +\delta_{ij}\tilde a_{ii}^{(0)} \right) Y_{0j}^{(k)},
\end{align}
where  $Y_{0i}^{(k+1)}$ may overwrite  $Y_{0i}^{(k)}$ immediately.
Convergence of the linear iteration \eqref{Nwtit} is governed by the matrix
\begin{align*}
 \Sc_0(z):=(\tilde A_0-zI)^{-1}(\tilde A_0-A_0),\ z=h_0\lambda\in\C,
\end{align*}
with eigenvalues $\lambda$ of the matrix $J_0$.
In the following subsections, we will present very small contraction factors $\rho_\R:=\sup\{\varrho(\Sc_0(z)):\,z\in(-\infty,0]\}$ along the negative real axis and $\rho_\alpha:=\sup\{\varrho(\Sc_0(z)):\,|\arg(z)-\pi|\le\alpha\}$ in the sector of $A(\alpha)$-stability of the corresponding methods.
\par
For linear test problems $y'=Jy$, $J\in\R^{m\times m}$, the starting step reduces to $Y_0=\Rc_0(h_0 J)y_0$, where
\begin{align*}
 \Rc_0(z)=(A_0-zI)^{-1}(a+zb)=\eins+ z(A_0-zI)^{-1}A_0\cc
\end{align*}
for methods satisfying the order conditions \eqref{OBstrt}.
The function $\Rc_0:\,\C\to\C^s$ is the stability function of the scalar problem $y'=\lambda y$, $y_0=1$, $z=h\lambda\in\C$.
\subsection{Method \texttt{IP2o3}}
With the increased order, already two-stage methods with order $s+1=3$ are of interest.
The order conditions from Lemma~\ref{LOrdsp1} lead to the Radau nodes $\cc\T=(\frac13,1)$, and the coefficients
\begin{align}\label{IP2o3}
 \bar K(\sigma)=\begin{pmatrix}
 \frac{2+\sigma}{6(1+\sigma)}&\cdot\\[2mm]
  \frac34&\frac14
 \end{pmatrix},\quad
 \bar B(\sigma)=\begin{pmatrix}
  \frac{-\sigma^2}{4(1+\sigma)}&\frac{(2+\sigma)^2}{4(1+\sigma)}\\[2mm]
  \cdot&1
 \end{pmatrix}
\end{align}
of method \texttt{IP2o3} (Initial-value Peer method with 2 stages and Order 3)
are uniquely determined.
In \eqref{IP2o3} and below, scalar zeroes are replaced by dots in order to highlight the structure.
The stability angle of the method is $\alpha\doteq77.87$ and $\tan\alpha\doteq 4.65$ is an upper bound for possible choices of the aperture $\eta$ in Lemma~\ref{LOmdefin}.
An additional column scaling of the basic weight $W_b$ as
\[ W=\begin{pmatrix} 1&\frac52\\1&\cdot\end{pmatrix}\]
leads to a nearly maximal aperture $\eta=17/6\doteq\tan(70.5^o)$ for $\sigma=11/10$.
For general $\sigma$, the matrices $M=W^{-1}\bar BW$ and $\Gamma=W^{-1}\bar KW$ turn out to be
\begin{align*}
 M=\begin{pmatrix}1&\cdot\\[2mm]
 \cdot&-\frac14\frac{\sigma^2}{1+\sigma} \end{pmatrix},
\ \Gamma=\begin{pmatrix}
  1&\frac{15}8\\[2mm]
 -\frac1{15}\frac{4+5\sigma}{1+\sigma}& -\frac1{12}\frac{5+7\sigma}{1+\sigma}\end{pmatrix}.
\end{align*}
It is plain that some fine-tuning of the choices for $\nu,\sigma$ and $\eta$ is required.
In the simplest case with no damping, $\nu=0$, and uniform grids, $\sigma=1$, the assumptions of the criterion in Theorem~\ref{TNStab} can be satisfied with a maximal value $\eta\cong3.2\doteq\tan(72.6^o)$.
Allowing for variable stepsizes with $\sigma_n\le1.1$ and $z=-\xi(1+\imag\eta),\,\xi\ge0,$ with a nearly maximal choice $\eta_0=17/6\doteq\tan(70.5^o)$ for $\nu=0$, the test matrix in \eqref{Omdef} becomes
\begin{align*}
 \Omega(\xi)=\begin{pmatrix}
 2\xi+\frac{93925}{2304}\xi^2 
 & \frac{793}{504}\xi+\imag\frac{18649}{3024}\xi-\frac{116675}{10368}\xi^2\\[2mm]
  \frac{793}{504}\xi-\imag\frac{18649}{3024}\xi-\frac{116675}{10368}\xi^2
 &\frac{690959}{705600}-\frac{127}{126}\xi+\frac{7119125}{2286144}\xi^2
\end{pmatrix}.
\end{align*}
Up to constant positive factors, the modified minors of $\Omega(\xi)$ are
\begin{align*}
\zeta_1(\xi)\circeq&\;4608+93925\xi, \\
\zeta_2(\xi)\circeq&\; 257899064832-342791943525\xi+72540000000\xi^2+40602250000\xi^3.
\end{align*}
The first one is obviously positive for $\xi>0$, see \eqref{Omg1}, and the Bernstein-P\'olya criterion shows positivity $\zeta_2(\xi)>0,\,\xi\ge0$, since $\zeta_2(0)>0$ and all coefficients are positive for $\hat\zeta_2(\xi)=(1+\xi)^{60}\zeta_2(\xi)$ with $\ell_2=60$.
For the reciprocal stepsize ratio $\sigma=10/11$, we have $\hat\zeta_2(\xi)>0$ already for $\ell_2\ge 9$ and with $\ell_2=16$ for $\sigma=1$.
For method \texttt{IP2o3}, this proves the bound
\begin{align*}
\ltnorm(I-z\bar K_n)^{-1}\bar B_n\rtnorm\le1,\ |\Im z|\le-\frac{17}6\Re z,\ \sigma_n\in{\cal S}=\left\{\frac{10}{11},1,\frac{11}{10}\right\}.
\end{align*}
Reducing the aperture to $\eta=2\doteq\tan(63.4^o)$, we may look for bounds with damping.
Now, for $\sigma=1.1$ a damping factor $\nu=1/20$ may be chosen leading to the test matrix
\begin{align*}
 \Omega(\xi)=\begin{pmatrix}
 \frac{19}{10}\xi+\frac{7221}{320}\xi^2
 & \frac{793}{504}\xi+\imag\frac{1097}{252}\xi-\frac{1795}{288}\xi^2\\[2mm]
 \frac{793}{504}\xi-\imag\frac{1097}{252}\xi-\frac{1795}{288}\xi^2
 &\frac{690959}{705600}-\frac{791849}{784000}\xi+\frac{876068231}{508032000}\xi^2\end{pmatrix},
\end{align*}
and the polynomials
\begin{align*}
\zeta_1(\xi)\circeq&\;608+7221\xi, \\
\zeta_2(\xi)\circeq&\; 302474211840-202787480336\xi+15917708856\xi^2+10919696051\xi^3.
\end{align*}
Here, the choice $\ell_2=36$ verifies that $\zeta_2(\xi)>0,\,\xi>0$.
Smaller exponents are sufficient for $\sigma=1$ ($\ell_2=13$) and $\sigma=10/11$ ($\ell_2=7$).
Hence, the criterion proves for the method \texttt{IP2o3} also that
\begin{align}\label{BIP2o3}
 \ltnorm(I-z\bar K_n)^{-1}\bar B_n\rtnorm\le\frac1{|1-z/20|},\ |\Im z|\le-2\,\Re z,
 \ \sigma_n\in{\cal S},
\end{align}
in the fixed norm \eqref{Wnorm} with the finite set ${\cal S}=\{10/11,1,1.1\}$.
\par
Considering now intervals ${\cal S}=[1-\hat\delta,1+\hat\delta]$, we mention that the common denominator of $\zeta_2(\xi)$ is $(2+\delta)^2>0$.
In order to prove \eqref{BIP2o3} with $\hat\delta$ near $0.1$, we apply the lower bound \eqref{zetadlt} to $\hat\zeta_2(\xi)$ with the largest exponent $\ell_2=36$ required for $\sigma=1.1$.
There is no room here to present this in detail.
In order to show the principle, we consider only the coefficients $\hat\zeta_{2,26}(\sigma)$ and $\hat\zeta_{2,27}(\sigma)$ since $\hat\zeta_{2,26}(1.1)<0$ was the only negative coefficient with the smaller exponent 35.
These seem to be the two critical functions for $\ell_2=36$, indeed, since all but the leading coefficient are negative for both, $\hat\zeta_{2kj}<0$, $j=1,\ldots,4$.
With the restriction $\delta\le\delta_0=0.1$, the floating point versions of \eqref{zetadlt} read
\begin{align*}
 \hat\zeta_{2,26}(1+\delta)\ge&\; 1.65\cdot 10^{8}-1.68\cdot 10^{9}|\delta|,\\
 \hat\zeta_{2,27}(1+\delta)\ge&\; 7.52\cdot 10^{7}-6.77\cdot 10^{8}|\delta|,
\end{align*}
leading to $\hat\delta\le\min\{0.098,0.111,\delta_0\}$.
\par
Since there is one order condition $A_0\cc^2=2\cc$ for the starting method only due to the additional term $b\otimes f(y_0)$, there exists a simple, diagonal solution $A_0=\diag(6,2)$.
For both stages, this is a step of the trapezoidal rule, each, with the corresponding stepsizes $h_0c_1=h_0/3$ and $h_0c_2=h_0$.
The local error of the trapezoidal rule is 3, of course.
Here, no iteration is required, $\rho_\R=\rho_\alpha=0$.
\subsection{Method \texttt{IP3o4}}
For $s=3$, the super-convergence condition \eqref{Supkv} may be solved by setting
$c_2=\frac12(1-2c_1)/(1-3c_1)$.
An $L(\alpha)$-stable 3-stage implicit LSRK method of order 4 with angle $\alpha=64.59^o$, $\tan\alpha\doteq2.1$, exists with the nodes $\cc=(\frac19,\frac7{12},1)\T$.
Its coefficients are given in Appendix~\ref{AIP3o4}, but we show the first entry of $\bar K(\sigma)$ here explicitly:
\begin{align*}
 \bar\kappa_{11}=\frac1{18}\frac{120+47\sigma+4\sigma^2}{60+47\sigma+6\sigma^2}
 \in\left[\frac1{14},\frac19\right],\ \sigma\in[0,2],
\end{align*}
in order to show that it is a smooth, positive and bounded function.
With an additional scaling of the last 2 columns of the weight matrix as 
\[  W=\begin{pmatrix}
  1&\frac{72}{13}&\cdot\\[2mm]
  1&\frac{17}{9}&\frac{3}2\\[2mm]
  1&\cdot&\cdot
  \end{pmatrix},\]
the norm criterion may be successfully applied with a maximal aperture of $\eta_0=8/5\doteq\tan(58.0^o)$ and no damping.
Here, $\zeta_2(\xi)>0,\,\xi\ge0,$ has only positive coefficients  for $\sigma\in{\cal S}=\{\frac{11}{12},1,\frac{12}{11}\}$.
The coefficients of $\zeta_3(\xi)$ have different signs (having up to 38 digits, they can not be displayed here).
However, $\hat\zeta_3(\xi)$ possesses positive coefficients only with $\ell_3=96$ ($\sigma=12/11$), $\ell_3=47$ ($\sigma=1$), and $\ell_3=59$ ($\sigma=11/12$). 
For method \texttt{IP3o4}, this proves the bound
\begin{align*}
\ltnorm(I-z\bar K_n)^{-1}\bar B_n\rtnorm\le1,\ |\Im z|\le-\frac85\Re z,\ \sigma_n\in{\cal S}=\Big\{\frac{11}{12},1,\frac{12}{11}\Big\}.
\end{align*}
On a smaller sector with $\eta=6/5\doteq\tan(50.2^o)$ the criterion allows for damping $\nu=1/64$ on the same set ${\cal S}$.
Again $\zeta_2(\xi)$ has positive coefficients, while $\hat\zeta_3(\xi)>0\,(\xi\ge0)$ requires exponents $\ell_3=110$ ($\sigma=12/11$), $\ell_3=150$ ($\sigma=1$) and $\ell_3=718$ ($\sigma=11/12$).
This also proves the bound
\begin{align*}
 \ltnorm(I-z\bar K_n)^{-1}\bar B_n\rtnorm\le\frac1{|1-z/64|},\ |\Im z|\le-\frac65\Re z,
  \ \sigma_n\in {\cal S}.
\end{align*}
The conditions for the starting step \eqref{Startab} are now $A_0(\cc^2,\cc^3)=(2\cc,3\cc^2)$.
The slack vector in \eqref{A0par} may be used to obtain $2+1$ block form of $A_0$, which fixes the upper $2\times 2$ block of $A_0$ completely.
Still, $A_0$ has positive eigenvalues $6,108/7,13/3$, and the Jacobian $A_0-h_0{\bf J}_0$ in the Newton step will be well-conditioned.
Optimization of the remaining (positive) parameters $\tilde a_{11}^{(0)},\tilde a_{22}^{(0)},\tilde a_{33}^{(0)}\equiv a_{33}^{(0)}$, lead to the values $\diag(\tilde A_0)=(16,\frac{17}3,\frac{13}3)$ and a rather small contraction factor $\rho_\R\doteq0.01248<1/80$ and $\rho_\alpha<1/63\doteq0.0159$.
The lower eigenvalue bound for Lemma~\ref{LKIJ} is $d_{\min}=13/3$.
All coefficients are presented in Appendix~\ref{AIP3o4}.
\subsection{Method \texttt{IP4o5}}
With $s=4$ stages, the search for methods with large stability angle even found $L$-stable methods with $\alpha=90^o$.
However, all these methods seem to have a second eigenvalue of $\bar B(1)$ very close to one violating the condition \eqref{kontrkt} for super-convergence and also have large entries in the stability matrix with norms $\|\bar B_n\|_\infty>30$.
Such methods may suffer from numerical instabilities, since $\|\bar B\|_\infty$ is an amplification factor for rounding errors in the step $Y_n=\bar B_nY_{n-1}+\ldots$.
\par
Instead we chose the following method \texttt{IP4o5} with $\alpha\doteq 71.2^o$ based on the nodes
\begin{align*}
 \cc\T=\left(\frac16,\frac37,\frac{31}{50},1\right).
\end{align*}
The coefficient matrix $\bar K(\sigma)$ is presented in Appendix~\ref{AIP4o5} and
\begin{align*}
 \bar\kappa_{11}=\frac1{12}\frac{6840+6363\sigma+1874\sigma^2+175\sigma^3}
 {3420+6363\sigma+2811\sigma^2+350\sigma^3}\in\left[\frac1{13},\frac16\right],\ \sigma\in[0,2],
\end{align*}
is again bounded away from zero.
With the weight matrix
\begin{align*}
 W=\begin{pmatrix}
  1&\frac{31}{20}&\cdot&\cdot\\[2mm]
  1&\frac{19}{20} &\frac{21}{10}&\cdot\\[2mm]
  1&\frac34&\frac{47}{25}&\frac13\\[2mm]
  1&\cdot&\cdot&\cdot
 \end{pmatrix},
\end{align*}
the norm criterion with $\nu=0$ is applicable up to the aperture $\eta_0=3/4\doteq\tan(36.8^o)$.
In this case, $\zeta_2(\xi)$ has only positive coefficients for $\sigma\in\{20/21,1,21/20\}$. 
The required exponents for producing positive coefficients of $\hat\zeta_i(\xi),\,i=3,4$, are 
\begin{align*}
\ell_3=7,\;\ell_4=33\;(\sigma=\frac{20}{21}),\;
\ell_3=7,\;\ell_4=53\;(\sigma=1),\;
\ell_3=10,\;\ell_4=367\;(\sigma=\frac{21}{20}).
\end{align*}
Reducing the aperture to $\eta=1/2\doteq\tan(26.5^o)$, an estimate with damping $\nu=1/280$ is established with the choices
\begin{align*}
\ell_3=4,\;\ell_4=39\;(\sigma=\frac{20}{21}),\;
\ell_3=6,\;\ell_4=63\;(\sigma=1),\;
\ell_3=7,\;\ell_4=360\;(\sigma=\frac{21}{20}).
\end{align*}
Again, $\zeta_2(\xi)$ has only positive coefficients.
Hence, for the method \texttt{IP4o5}, the following bounds hold:
\begin{align*}
\ltnorm(I-z\bar K_n)^{-1}\bar B_n\rtnorm\le1,&\ |\Im z|\le-\frac34\Re z,\ \sigma_n\in{\cal S}=\Big\{\frac{20}{21},1,\frac{21}{20}\Big\},\\\notag
 \ltnorm(I-z\bar K_n)^{-1}\bar B_n\rtnorm\le\frac1{|1-z/280|},&\ |\Im z|\le-\frac12\Re z,
  \ \sigma_n\in{\cal S}.
\end{align*}
For \texttt{IP4o5}, the contraction is very sensitive to the choice of parameters and a small factor $\rho_\R\doteq0.03664<1/27$ could only be found by computer searches.
The contraction factor in the whole sector is still good, $\rho_\alpha<0.09$.
Since we may require rigorous norm bounds for the starting method, we chose rational values for $A_0e_4=(-\frac1{25},\frac{32}{323},\frac3{25},\frac{105}{16})\T$ leading to eigenvalues satisfying $d_{\min}\ge5.4$.
The diagonals of $\tilde A_0$ are given as double precision numbers in Appendix~\ref{AIP4o5}.
\begin{table}
\centering
\centerline{\begin{tabular}{|l|c|c|c|c|c|c|c|c|}\hline
  name &$s$&$r,q$ & $\alpha$ & $\bar\sigma$&$\vartheta_0$&$\vartheta,\nu$&$err_{r}$\\\hline
  IP2o3&$2$&$3,0$&$77.87^o$&1.10&$70.5^o$&$63.4^o,\frac1{20}$&4.0e-3\\[1mm]
  IP3o4&$3$&$4,0$&$64.59^o$&$\frac{12}{11}$& $58.0^o$&$50.2^o,\frac1{64}$&4.5e-3\\[1mm]
  IP4o5&$4$&$5,0$&$71.20^o$&1.05& $36.8^o$&$26.5^o,\frac1{280}$&4.9e-4\\[1mm]\hline
  AP4o33vgi&$4$&$3,3$&$61.59^o$&$\frac{10}9$&$59.5^o$&$54.4^o,\frac1{40}$&1.3e-2\\[1mm]
\hline
\end{tabular}}
\parbox{14cm}{
\caption{Properties of the Peer methods: order $r$, adjoint order $q$, stability angle $\alpha$, maximal stepsize ratio $\bar\sigma$, $\vartheta=\arctan(\eta)$, damping factor $\nu$, and error constant \eqref{errr}.
}\label{TPT}
}
\end{table}

\section{The global error for stiff initial value problems}\label{GerrIVP}
We show the straight-forward application of the norm bound to error estimates for semi-linear problems \eqref{semilin}, where $J$ may have large norm but is assumed to have an eigen-decomposition with eigenvector basis matrix $X\in\C^{s\times s}$,
\begin{align}\label{XJX}
 X^{-1}JX=\Lambda=\diag_j(\lambda_j),\ \mu:=\max_{j=1}^m\Re\lambda_j<0,
\end{align}
and $g$ to be Lipschitz continuous
\begin{align}\label{Lipg}
 \|g(u)-g(v)\|_2\le L_g\|u-v\|_2,\quad u,v\in\R^m,
\end{align}
in the Euclidean norm for simplicity.
Denoting the stage vectors of the exact solution $y^\star$ of \eqref{semilin} and its derivative by ${\bf y}_n:=\big(y^\star(t_n+h_nc_j)\big)_{j=1}^s$, ${\bf y}_n':=\big((y^\star)'(t_n+h_nc_j)\big)_{j=1}^s$, the local errors of the Peer method are defined by
\begin{align*}
 \tau_n:={\bf y}_n-\bar B_n{\bf y}_{n-1}-h_n\bar K_n{\bf y}_n',\ n=1,\ldots,N.
\end{align*}
Then, subtracting this equation from \eqref{Peerivp}, the errors $\check Y_n:=Y_n-{\bf y}_n$ satisfy the recursion $\check Y_n=\bar B_n\check Y_{n-1}+h_n\bar K_n\big(J\check Y_n+g({\bf y}_n+\check Y_n)-g({\bf y}_n)\big)-\tau_n$.
Since the diagonal elements of $\bar K_n$ are positive and $\mu<0$ in \eqref{XJX}, this is equivalent with
\begin{align}\label{FeGl}
 \check Y_n=&\;(I-h_n\bar K_n\otimes J)^{-1}(\bar B_n\check Y_{n-1}-\tau_n)\nonumber\\
 &+h_n(\bar K_n^{-1}-h_nI\otimes J)^{-1}\big(G({\bf y}_n+\check Y_n)-G({\bf y}_n)\big).
\end{align}
Dealing with the nonlinear part of the error equation \eqref{FeGl} requires a bound for the matrix multiplying it.
In order to apply the weighted norm estimate to the full differential equation, we have to extend the definition of the norm to vectors and matrices in dimension $s\cdot m$.
Generalizing the old notation, we define
\begin{align}\label{WXnorm}
 \ltnorm \Theta\rtnorm :=\|(W\otimes X)^{-1}\Theta(W\otimes X)\|_2,\ \Theta\in\C^{(sm)\times(sm)},
\end{align}
which is induced by the vector norm $\ltnorm Y\rtnorm:=\|(W\otimes X)^{-1}Y\|_2,\,Y\in\C^{sm}$.
\begin{lemma}\label{LKIJ}
For $\sigma\in[\underline\sigma,\bar\sigma]\subset(0,\infty)$ let $A=A(\sigma)\in\R^{s\times s}$ be a uniformly bounded family of matrices having eigenvalues that satisfy $\Re\lambda_i(A(\sigma))\ge d_{\min}>0$, $i=1,\ldots,s$.
If $D_A(\sigma):=\diag_i(\lambda_i(A))$ contains the eigenvalues of $A$, then with $J$ from \eqref{XJX} it holds that
\begin{align*}
 \ltnorm (A-hI\otimes J)^{-1}\rtnorm\le
 \sum_{j=0}^{s-1}\frac{\|W^{-1}AW-U D_AU^\ast\|_2^j}{d_{min}^{j+1}}=:\theta,
\end{align*}
where $U$ is the unitary matrix from the Schur form $W^{-1}AW=U RU^\ast$.
\end{lemma}
\par\noindent
{\bf Proof}
Since the eigenvector basis $X^{-1}$ of $J$ is one factor of the norm weight, we have
\begin{align*}
 \ltnorm (A-h I\otimes J)^{-1}\rtnorm
 = \|(W^{-1}AW-h I\otimes\Lambda)^{-1}\|_2
 = \max_{k=1}^m\|(W^{-1}AW-z_kI)^{-1}\|_2,
\end{align*}
$z_k=h\lambda_k(J)$.
Since $D_A$ is also the main diagonal of the Schur matrix $R$, the difference $W^{-1}AW-UD_AU^\ast $ is nilpotent and the following Neumann series becomes a polynomial while the unitary factor $U$ may be omitted sometimes in the spectral norm,
\begin{align*}
 &\|(W^{-1}AW-z_kI)^{-1}\|_2\\
 &\le
 \|(D_A-z_kI)^{-1}\|_2\sum_{j=0}^{s-1}\|(W^{-1}AW-UD_AU^*)(UD_AU^*-z_kI)^{-1}\|_2^j\\
 &\le\sum_{j=0}^{s-1}\|W^{-1}AW-UD_AU^*\|_2^j\|(D_A-z_kI)^{-1}\|_2^{j+1}
\end{align*}
The diagonals in $D_A-z_kI$ are bounded from below by $|\lambda_i(A)-z_k|\ge d_{\min}-h\mu\ge d_{\min}$.
\qed\\
\par\noindent
We note that the bound from the lemma will be applied now with the triangular matrices $A_n=\bar K_n^{-1}$, where $U $ may be chosen as the orthogonal factor in the QR decomposition of $W\T =U R$.
The following Theorem essentially is a global stability estimate only for perturbations of the Peer step \eqref{Peerivp} relating to the solution $y^\star$.
\begin{theorem}\label{TAWPKvg}
Consider a Peer two-step method \eqref{Peerivp} having lower triangular coefficient matrix $\bar K_n(\sigma)$ and positive diagonals bounded as  $\max_{i=1}^s\bar\kappa_{ii}(\sigma)\le1/d_{\min}$ and coefficient $\bar B(\sigma)$ with bounded entries for $\sigma\in[\underline{\sigma},\bar\sigma]\subset(0,2]$.
Assume that there exists a constant weight matrix $W\in\R^{s\times s}$ such that for some appropriate $\nu,\eta\ge0$ the matrix family $\Omega(\xi)$ in \eqref{Omdef} with matrices $M=W\T\bar B(\sigma)W$ and $\Gamma=W^{-1}\bar K(\sigma)W$ is positive semidefinite for $\xi\ge0$,  $\sigma\in\Sc\subseteq[\underline{\sigma},\bar\sigma]$ according to the criterion of Theorem~\ref{TNStab}.
Let the right-hand side of the differential equation \eqref{semilin} satisfy the assumptions \eqref{XJX} and \eqref{Lipg} where the eigenvalues of $J$ are restricted by $|\Im\lambda_i(J)|\le-\eta\Re\lambda_i(J)$, $i=1,\ldots,m$.
\par
Then, with some starting vector $Y_0\in\R^{sm}$ and grids with stepsize ratios $\sigma_n\in\Sc$ and stepsizes $h_n\le1/(2 L_\phi)$, where $L_\Phi:=\theta L_g cond_2(W\otimes X)$, the Peer solutions $Y_n$, $n\ge1$, exist.
The errors are bounded by
\begin{align}\label{globFe}
 \ltnorm Y_n-{\bf y}_n\rtnorm\le
 e^{\varpi(t_{n+1}-t_1)}\ltnorm Y_0-{\bf y}_0\rtnorm
  +C\sum_{k=1}^n e^{\varpi(t_{n+1}-t_k)}\ltnorm\tau_{k}\rtnorm,\ n\ge1, 
\end{align}
with residuals $\tau_n$ in \eqref{lokfe0} and constants
 $\varpi:=0.9\nu\mu+2L_\Phi$ and $C$.
\end{theorem}
{\bf Proof}
With $r_n:=(I-h_n\bar K_n\otimes J)^{-1}(\bar B_n\check Y_{n-1}-\tau_n)$, $n\ge1$, we may write the error equation \eqref{FeGl} as fixed point equation $\check Y_n=r_n+h_n\Phi_n(\check Y_{n})$, where $\Phi_n(0)=0$.
Considering two arguments $\hat Y,\tilde Y\in\R^{sm}$, we get
\begin{align*}
\ltnorm\Phi_n(\hat Y)-\Phi_n(\tilde Y)\rtnorm
=&\;\|(W\otimes X)^{-1}(\bar K_n^{-1}-h_nI\otimes J)^{-1}\big(G({\bf y}_n+\hat Y)-G({\bf y}_n+\tilde Y)\big)\|_2\\
 \le&\;\theta \|W^{-1}\|_2\|X^{-1}\|_2L_g\|\hat Y-\tilde Y\|_2\\
 \le&\;\theta L_g \,cond(W\otimes X)\ltnorm\hat Y-\tilde Y\rtnorm,
\end{align*}
with the bound $\ltnorm(\bar K_n^{-1}-h_nI\otimes J)^{-1}\rtnorm\le\theta$ from Lemma~\ref{LKIJ}.
Hence, with $L_\Phi:=\theta L_g\, cond(W\otimes X)$ the map $Y\mapsto r_n+h_n\Phi_n(Y)$ has Lipschitz constant $h_nL_\Phi$ and is contractive for $h_nL_\Phi\le1/2$, for instance.
Hence, if $h_n$ is small enough, each closed $\varepsilon$-neighborhood of the origin with $\varepsilon\ge2\ltnorm r_n\rtnorm$ is mapped onto itself and there exists a unique fixed point $\check Y_n$ in that neighborhood satisfying the a-priori bound
\begin{align*}
\ltnorm \check Y_n\rtnorm=&\;2\ltnorm r_n+h_n\big(\Phi_n(\check Y_n)-\Phi_n(0)\big)\rtnorm-\ltnorm\check Y_n\rtnorm\\
\le&\; 2\ltnorm r_n\rtnorm +(h_nL_\Phi-1)\ltnorm\check Y_n\rtnorm
\le 2\ltnorm r_n\rtnorm.
\end{align*}
Refining this argument, we obtain with Theorem~\ref{TNStab} that
\begin{align*}
\ltnorm \check Y_n\rtnorm=&\;\ltnorm r_n+h_n\big(\Phi_n(\check Y_n)-\Phi_n(0)\big)\rtnorm
\le (1+2h_nL_\Phi)\ltnorm r_n\rtnorm\\
 =&\;(1+2h_nL_\Phi)\ltnorm (I-h_n\bar K_n\otimes J)^{-1}(\bar B_n\check Y_{n-1}-\tau_n)\rtnorm\\
\le&\;(1+2h_nL_\Phi)\Big(\frac1{1-h_n\nu\mu}\ltnorm\check Y_{n-1}\rtnorm
 +\theta\ltnorm \bar K_n^{-1}\rtnorm\,\ltnorm\tau_n\rtnorm\Big),
\end{align*}
with the upper bound $\theta$ from Lemma~\ref{LKIJ}.
Products of factors $(1+2h_nL_\Phi)/(1-h_n\nu\mu)$ may be simplified by use of the exponential function.
For the numerator, we use the simple bound
\[ \frac1{1+x}\le e^{-0.9x},\quad x\in[0,0.2].\]
Hence for $-0.2\le h_n\nu\mu\le0$, we obtain the simplified recursive estimate
\begin{align*}
\ltnorm \check Y_n\rtnorm\le e^{h_n(0.9\nu\mu+2L_\Phi)}\ltnorm\check Y_{n-1}\rtnorm+C_n\ltnorm\tau_n\rtnorm
\end{align*}
with bounded constants $C_n=(1+2h_nL_\Phi)\theta\ltnorm K_n^{-1}\rtnorm\le C$.
The assertion \eqref{globFe} follows inductively.
\qed
\begin{remark}
We have computed all constants in these estimates explicitly in order to show that none depends on the stiffness $\|J\|_2$, at least if $cond_2(X)$ is of moderate size, which is true, for instance, if $J$ is a normal matrix.
\par
Also, for some problems with small Lipschitz constants $L_g$ or $\mu\ll0$, the estimate \eqref{globFe} may even show error decay if $\varpi=0.9\nu\mu+2L_\Phi<0$.
\end{remark}
If the Peer method satisfies the conditions for local order $s+1$ from Lemma~\ref{LOrdsp1}, the estimate \eqref{globFe} from the Theorem directly leads to the bound
\begin{align}\label{Konvos}
 \ltnorm Y_n-{\bf y}_n\rtnorm\le
 e^{\varpi(t_{n+1}-t_1)}\ltnorm Y_0-{\bf y}_0\rtnorm
  +C\max_{k=1}^n e^{\varpi(t_{n+1}-t_k)}h_k^s\|{y^\star}^{(s+1)}\|_{[k]}.
\end{align}
By summation of local error bounds, this estimate shows convergence of order $s$ only for general grids satisfying $\sigma_n\in\Sc$.
The estimate \eqref{Konvos} relates the stepsize to the local size of the derivative which we need in more general form in the next theorem.
Hence, we define
\begin{align*}
 \|y^{(r)}\|_{[n]}:=\max_{t\in[t_n,t_{n+1}]}\|y^{(r)}(t)\|_2,\quad
 \|y^{(r_1:r_2)}\|_{[n]}:=\max_{r_1\le k\le r_2}\|y^{(k)}(t)\|_{[n]},
\end{align*}
with $r_1\le r_2$.
In order to recover global order $s+1$, the condition \eqref{Supkv} for super-convergence has to be used.
The improved order can be shown in the stiff context for smooth grids by a refinement of a technique from \cite{SchneiderLangWeiner2021}.
However, the condition \eqref{Supkv}, $e_s\T\beta_{s+1}(\sigma)\equiv0$, is not sufficient to see super-convergence in theory and practice.
In addition, the subdominant eigenvalues of $\bar B(\sigma)$ have to be smaller than one.
A corresponding assumption on the southeast block of the block decomposition \eqref{blkivp}, $W^{-1}\bar B(\sigma)W=1\oplus B_{se}(\sigma)$, was already given in \eqref{kontrkt} with a fairly small damping factor $\gamma\cong 0.8$ in order to see super-convergence reliably in practice.
Now, for the proofs a bound on the norm 
\begin{align}\label{kontrktn}
 \|B_{se}(\sigma)\|_2\le\tilde\gamma<1
\end{align}
is required where the choice of $\tilde\gamma<1$ only affects the constants in the estimates, see \cite{LangSchmitt2026a}.
\par
The proof of super-convergence will use a different solution $\tilde Y_{ni}$ modified by multiples of $(y^\star)^{(s+1)}(t_n)$, and in the corresponding local error, products with the Jacobian of $f$ will appear in the form
$\bar f_{y,ni}\cdot (y^\star)^{(s+1)}(t_n)$ with integral means $\bar f_{y,ni}=\int_0^1 f_y\big((1-\xi)Y_{ni}+\xi\tilde Y_{ni})d\xi$.
In the stiff context, using norms of these matrices would severely spoil the constants in the error estimates.
However, for semi-linear problems \eqref{semilin}, the highest solution derivative obeys
\[ (y^\star)^{(s+2)}=\big(J+g_y(y^\star)\big)(y^\star)^{(s+1)}+\ldots\]
where the remaining terms depend on the derivatives $(y^\star)',\ldots,(y^\star)^{(s)}$ and those of $g$ but not on $J$.
Hence, we may expect that $\bar f_{y,ni}\cdot (y^\star)^{(s+1)}(t_n)\cong (y^\star)^{(s+2)}(t_n)$ and we will use for simplicity the following technical assumption:
\begin{align}\label{FYysp}
 \|f_y(u)(y^\star)^{(s+1)}(t_n)\|\le const\|(y^\star)^{(1:s+2)}\|_{[n]},
 \ \|u-{y^\star}(t_n)\|\le\varepsilon.
\end{align}
\begin{theorem}\label{Tsupknv}
Let the solution of \eqref{semilin} be smooth, $y^\star\in C^{s+2}[0,T]$, and assume \eqref{FYysp}.
Let the LSRK Peer method satisfy the assumptions from Theorem~\ref{TAWPKvg}, the conditions from Lemma~\ref{LOrdsp1} for local order $s+1$, and the super-convergence condition \eqref{Supkv}.
Let $K_n=K(\sigma_n)$ depend smoothly on $\sigma$ in a neighborhood of $\sigma=1$, let $W^{-1}\bar BW$ satisfy \eqref{kontrktn}, and let the grid be smooth with $\sigma_n=1+\eta_n h_n$, $\eta_n\le\bar\eta,\,1\le n\le N$.
Then, the global error is of order $s+1$ and satisfies the estimate
\begin{align*}
\ltnorm Y_n-{\bf y}_n\rtnorm\le
e^{\varpi(t_{n+1}-t_1)}\ltnorm Y_0-{\bf y}_0\rtnorm
  +C\max_{k=1}^n h_k^{s+1}\|(y^\star)^{(1:s+2)}\|_{[k]},\ n\ge1.
\end{align*}
\end{theorem}
{\bf Proof}
We recall the form of the local error \eqref{lokfer} as $\tau_n=h_n^{s+1}\beta_{s+1}(\sigma_n)\otimes(y^\star)^{(s+1)}(t_n)+O(h_n^{s+2})$.
Now, the technique from \cite{SchneiderLangWeiner2021} considers a modified solution
\begin{align*}
\tilde Y_n:=Y_n-h_n^{s+1} v_n\otimes (y^\star)^{(s+1)}(t_n),\quad
(I-\bar B_n)v_n=\beta_{s+1}(\sigma_n),\; n\ge1.
\end{align*}
The tensor product has been used here in order to show the meaning more clearly, but it will be omitted below for convenience.
Obviously, the system for $v_n$ is singular, but according to the LSRK form and Lemma~\ref{LOrdsp1}, it is solvable for arbitrary $\sigma$ since $e_s\T(I-\bar B)=e_s\T\beta_{s+1}(\sigma)\equiv0$.
In fact, with block decomposition according to Lemma~\ref{LWb} we see for $\breve v_n=W^{-1}v_n$ that
\begin{align*}
 (I-W^{-1}\bar B_nW)\breve{v}_n
 =\begin{pmatrix}0&0\T\\[1mm]
 0&I-B_{se}(\sigma_n)\end{pmatrix}
  \begin{pmatrix}\breve v_{n,1}\\[1mm]
  \breve v_n^\perp \end{pmatrix}
  =\begin{pmatrix}0\\(W^{-1}\beta_{r+1}(\sigma_n))^\perp  
  \end{pmatrix},
\end{align*}
which may be solved by $\breve v_{n1}=0$, $v_n^\perp=(I-B_{se}(\sigma_n))^{-1}(W^{-1}\beta_{r+1}(\sigma_n))^\perp$, leading to $\|v_n\|=\|W\breve v_n\|_2\le cond(W)\|\beta_{s+1}(\sigma_n)\|_2/(1-\tilde\gamma)$ according to \eqref{kontrktn}.
\par
We will show now that the truncation error of $\tilde Y_n$ satisfies $\tilde\tau_n=O(h_n^{s+2})$, but some preliminary estimates are needed for that.
First, by \eqref{Bsig} and \eqref{betar} both $\bar B(\sigma)$ and $\beta_{s+1}(\sigma)$ vary smoothly in intervals $\sigma\in[\underline{\sigma},\bar\sigma]$, $\underline{\sigma}<1<\bar\sigma$, where $K(\sigma)$ ist smooth.
Hence, for smooth grids, differences of these objects evaluated at $\sigma_n$ and $\sigma_{n-1}$ are bounded by multiples of
\begin{align}\label{sminuss}
 |\sigma_n-\sigma_{n-1}|\le|\eta_n h_n+\eta_{n-1}h_{n-1}|
 \le \bar\eta(1+\underline{\sigma}^{-1})h_n.
\end{align}
Since differences of matrix inverses satisfy $\Theta_1^{-1}-\Theta_2^{-1}=\Theta_1^{-1}(\Theta_2-\Theta_1)\Theta_2^{-1}$, this leads to smooth dependence of the vectors $v_n$ itself,
\begin{align*}
 &\|\breve v_n-\breve v_{n-1}\|_2\\
 &=\|(I-B_{se}(\sigma_n))^{-1}(W^{-1}\beta_{s+1}(\sigma_n))^\perp
  -(I-B_{se}(\sigma_{n-1}))^{-1}(W^{-1}\beta_{s+1}(\sigma_{n-1}))^\perp\|_2\\
  &\le\frac{C}{(1-\tilde\gamma)^2}|\sigma_n-\sigma_{n-1}|=O(h_n)
\end{align*}
by \eqref{sminuss}.
For smooth stepsize ratios from $[\underline{\sigma},\bar\sigma]$, it also holds that
\[h_n^{s+1}-h_{n-1}^{s+1}=h_n^{s+1}(1-(1+\eta_n h_n)^{-s-1}))=O(h_n^{s+2}).\]
Now, with these preliminaries, we get for the modified truncation error that
\begin{align}\notag
\tilde\tau_n=&\;\tilde Y_n-\bar B_n\tilde Y_{n-1}-h_n\bar K_n F(\tilde Y_n)
\\\notag
 =&\;\tau_n-h_n^{s+1}v_n(y^\star)^{(s+1)}(t_n)+h_{n-1}^{s+1}\bar B_nv_{n-1}(y^\star)^{(s+1)}(t_{n-1})\\\label{FYf}
 &+h_n^{s+2}\bar K_nF_Y\cdot v_n(y^\star)^{(s+1)}(t_n)\\\notag
 =&\;\underbrace{\tau_n-h_n^{s+1}(I-\bar B_n)v_n(y^\star)^{(s+1)}(t_n)}_{=O(h_n^{s+2}\|(y^\star)^{(s+2)}\|_{[n]})}
 -(h_n^{s+1}-h_{n-1}^{s+1})\bar B_nv_n(y^\star)^{(s+1)}(t_n)\\\notag
 &\;-h_{n-1}^{s+1}\bar B_n\big(v_n(y^\star)^{(s+1)}(t_n)-v_{n-1}(y^\star)^{(s+1)}(t_{n-1}\big))
 +O(h_n^{s+2}).
\end{align}
The factor $F_Y$ in \eqref{FYf} is a block-diagonal matrix with diagonal blocks $\bar f_{y,ni}$ defined above.
Since we know by Theorem~\ref{TAWPKvg} that $Y_n$ is an $O(h^s)$ approximation of ${\bf y}_n$, the derivative $f_y$ is evaluated very close to ${y^\star}(t_{ni})$, where $t_{ni}-t_n=h_nc_i$, and for $h_n$ small enough, with assumption \eqref{FYysp} the term $F_Y\cdot v_n{y^\star}^{(s+1)}(t_n)$ may be bounded by a constant times $\|(y^\star)^{(1:s+2)}\|_{[n]}$.
Then, with $\ltnorm\bar B_n\rtnorm=1$ and $(y^\star)^{(s+1)}(t_{n})-(y^\star)^{(s+1)}(t_{n-1}))=O(h_n\|(y^\star)^{(s+2)}\|_{[n-1]})$, we obtain for the weighted residual that
\begin{align*}
 \ltnorm\tilde\tau_n\rtnorm\le&\;
 \frac{h_{n}^{s+1}}{\sigma_n^{s+1}}\|\breve v_n-\breve v_{n-1}\|\|X^{-1}{y^\star}^{(s+1)}(t_{n})\|
 +Ch_n^{s+2}\big(\|(y^\star)^{(s+2)}\|_{[n-1]}+\|(y^\star)^{(1:s+2)}\|_{[n]}\big)\\
 \le&\;Ch_n^{s+2}\big(\|(y^\star)^{(s+2)}\|_{[n-1]}+\|(y^\star)^{(1:s+2)}\|_{[n]}\big).
\end{align*}
Applying a slightly simplified version of the bound \eqref{globFe}, we get
\begin{align*}
 \ltnorm Y_n-{\bf y}_n\rtnorm
 \le&\; \ltnorm\tilde Y_n-{\bf y}_n\rtnorm+\ltnorm Y_n-\tilde Y_n\rtnorm\\
 \le&\;e^{\varpi(t_{n+1}-t_1)}\ltnorm Y_0-{\bf y}_0\rtnorm
  +C\sum_{k=1}^n \ltnorm\tilde\tau_{k}\rtnorm_2+h_n^{s+1}\|{y^\star}^{(s+1)}\|_{[k]}\\
 \le&\;e^{\varpi(t_{n+1}-t_1)}\ltnorm Y_0-{\bf y}_0\rtnorm
  +C\max_{k=1}^n h_k^{s+1}\|{y^\star}^{(1:s+2}\|_{[k]}.
\qed
\end{align*}
Theorem~\ref{Tsupknv} may be refined, if the starting procedure \eqref{Startab} is used.
\begin{lemma}
Let the assumptions of Theorem~\ref{Tsupknv} be satisfied.
Assume that the matrix $A_0$ in the starting step \eqref{Startab} possesses only eigenvalues with positive real part, $\Re\lambda_i(A)\ge d_{\min}>0$, and satisfies the order conditions \eqref{OBSred} for local order $s+1$.
Then, a simplified bound for the global error of the Peer two-step method is
\begin{align*}
\ltnorm Y_n-{\bf y}_n\rtnorm\le
 C\max_{k=0}^n h_k^{s+1}\|{y^\star}^{(1:s+2)}\|_{[k]},\ n\ge0.
\end{align*}
\end{lemma}
{\bf Proof}
According to \eqref{lokfe0}, the error equation for the start \eqref{Startab} is
\begin{align*}
 \check Y_0=(A_0-h_0{\bf J}_0)^{-1}\big(h_0(g({\bf y}_0+\check Y_0)-g({\bf y}_0))-A_0\tau_0\big),
\end{align*}
which may also be written as $\check Y_0=r_0+h_0\Phi_0(\check Y_0)$.
Since $A_0$ is no longer triangular, its Schur form is used in Lemma~\ref{LKIJ}, leading again to a Lipschitz condition for $h_0\Phi_0$ with constant $h_0L_\Phi=h_0\theta L_g$.
For small $h_0\le 1/(2L_\Phi)$, also $\|\tau_0\|=O(h_0^{s+1}\|y^{(s+1)}\|_{[0]})$ will be small enough and we know that a unique solution $\check Y_0$ exists satisfying $\ltnorm\check Y_0\rtnorm\le2\ltnorm r_0\rtnorm\le Ch_0^{s+1}\|y^{(s+1)}\|_{[0]}$.
\qed\\
\par\noindent
We note that these Theorems and the Lemma apply to all our methods with the data given in Table~\ref{TPT}.
\section{Boundary value problems in optimal control}\label{SAPOC}
In ODE constrained optimal control problems there is an additional control variable $u(t)\in\R^d$ in the right-hand side of the differential equation like $\hat f(y,u)=Jy+\hat g(y,u)$, and the minimum of some objective function $\CC(y(T))$ of the final state over a closed and convex admissible control set $U_{ad}\in\R^d$ is sought for.
The Karush-Kuhn-Tucker (KKT) optimality conditions lead to an additional adjoint differential equation for a Lagrange multiplier $p(t)\in\R^m$.
Under suitable assumptions, the control can be recovered from the other variables in the form $u(y,p)$ and eliminated formally leading to a two-point boundary value problem on $[0,T]$ of the form 
\begin{align}\label{AWPyp}
 y'(t)=&\;Jy(t)+g\big(y(t),p(t)\big),\ y(0)=y_0,\\\label{ADglyp}
 p'(t)=&\;-J\T p(t)+\psi\big(y(t),p(t)\big),\ p(T)=\nabla_y\CC\big(y(T)\big)\T,
\end{align}
see, e.g., \cite{LangSchmitt2025a}. The functions $g,\psi$ are defined by
\begin{align*}
 g(y,p)=\hat g\big(y,u(y,p)\big),\quad 
 \psi(y,p)=-\nabla_y\hat g\big(y,u(y,p)\big)\T p.
\end{align*}
$(y^\star,p^\star)$ will be the exact solution of \eqref{AWPyp}, \eqref{ADglyp}.
For simplicity, we will treat this problem formulation here, since the discussion of the original problem is much more elaborate, see \cite{LangSchmitt2023,LangSchmitt2026a}.
In the FDTO approach ({\em First Discretize Then Optimize}), the KKT conditions for the discrete initial value problem with the Peer method are considered, which lead to adjoint time steps for discrete Lagrange multipliers $P_n=\big(P_{nj}\big)_{j=1}^s\in\R^{sm}$.
Also the adjoint Peer steps should lead to good approximations for the adjoint differential equation \eqref{ADglyp}.
Hence, the Peer methods have to satisfy additional adjoint order conditions.
This additional demand may be ameliorated by considering redundant formulations of the Peer methods with three coefficient matrices $(A,B_n,K)$ for steps from the interior of the grid, where $K\in\R^{s\times s}$ is a constant positive definite diagonal matrix here.
In addition, slightly different methods are used with coefficients $(A_0,K)$ for the start and $(A_N,B_N,K)$ for the end step.
So, the discrete boundary value problem considered in \cite{LangSchmitt2025a} becomes
\begin{align}\label{VPMstrt}
A_0Y_0=&\;a\otimes y_0+h_0 K(JY_0+G(Y_0,P_0)),\\\label{VPMstd}
A_nY_n=&\;B_nY_{n-1}+h_nK(JY_n+G(Y_n,P_n)),\ 1\le n\le N,\\\label{APMstd}
A_n\T P_n=&\;B_{n+1}\T P_{n+1}+h_nK(J\T P_n-\Psi(Y_n,P_n)),\ 0\le n\le N-1,\\\label{APMend}
A_N\T P_N=&\;w \otimes\nabla_y\CC(y_h(T))\T+h_NK(J\T P_N-\Psi(Y_N,P_N)).
\end{align}
Here, $a=A_0\eins$ and the approximation at the end point is $y_h(T)=\sum_{i=1}^s w_i Y_{Ni}$ with $w=A_N\T\eins$.
For the Dahlquist test equation, the adjoint step \eqref{APMstd} reduces to $P_n=\Rc_{n+1}^\dagger(z)P_{n+1}$ with the stability function
\begin{align}\label{Stfna}
 \Rc_{n+1}^\dagger(z)=(A_n\T-zK)^{-1}B_{n+1}\T
 =\big(B_{n+1}A^{-1}(I-zKA_n^{-1})^{-1}\big)\T,
\end{align}
$0\le n\le N-1$. Since $\Rc_{n+1}^\dagger(z)$ is related to $\Rc_{n+1}(z)$ by transposition and a commutation of factors, both possess the same eigenvalues.
Relations for the spectral norm, however, are not as obvious.
\par
The truncation error of order $q\le s$ of the adjoint step \eqref{APMstd} is
\begin{align}\label{lokfead}
 \tau_n^P=&\;h_n^q\beta_q^\dagger(\sigma_{n+1})p^{(q)}(t_n)+O(h_n^{q+1}\|p^{(q+1)}\|_{[n]}),\\\notag
 \beta_q^\dagger(\sigma):=&\;\frac1{q!}A_n\mT(A_n\T\cc^q-B_{n+1}\T(\eins+\sigma\cc)^q+qK\cc^{q-1}),
\end{align}
if the terms of lower order vanish, i.e., $\beta_1^\dagger=\ldots=\beta_{q-1}^\dagger\equiv0$, \cite{LangSchmitt2025a}.
For the adjoint end step \eqref{APMend}, the corresponding expressions are
\begin{align}\label{lokfean}
 \tau_N^P=&\;h_N^q\beta_{q,N}^\dagger p^{(q)}(t_N)+O(h_N^{q+1}\|p^{(q+1)}\|_{[N]}),\\\notag
 \beta_{q,N}^\dagger:=&\;\frac1{q!}(\cc^q+A_n\mT(qK\cc^{q-1}-w)\big).
\end{align}
For general $\sigma_n$, the combination with the forward conditions \eqref{Ordvr} leads to strong restrictions on all coefficients of the method and may be satisfied only if the matrix
\begin{align}\label{Qqr}
 \QQ_{q,r}:=V_q\T B(\sigma)V_r\PP_r^{-1}=e_1e_1\T\in \R^{q\times r}
\end{align}
related to $B(\sigma)$ is independent of $\sigma$, \cite{LangSchmitt2025a}.
The other coefficients $A$ and $K$ are also fixed to a large extend.
Hence, in order to leave room in $B(\sigma)$ for $\sigma$-dependency, the order $q$ of {\em adjoint Peer methods} is below $s$.
In fact, for the \textit{Pulcherrima Triplet} below, $q=r=s-1$ holds.
\subsection{Norms for the Pulcherrima Triplet \texttt{AP4o33vgi}}
We will only discuss the extraordinary  triplet \texttt{AP4o33vgi} from  \cite{LangSchmitt2025a,LangSchmitt2026a} being its own adjoint method, which means that the whole triplet (including the boundary steps) is invariant under a flip of the integration interval $[0,T]$.
It is based on the nodes $\cc\T=(0,\frac13,\frac23,1)$ and the last stage in LSRK form uses the weights $e_4\T\bar K=\frac18(1,3,3,1)$ of the pulcherrima quadrature rule.
Its standard method is $L(\alpha)$-stable with $\alpha=61.69^o$ and $\tan\alpha\doteq1.85$.
\par
As a consequence of \eqref{Qqr}, its coefficient matrices are fixed to a large extend.
Of relevance for the convergence analysis below is the form of the congruent matrix
\begin{align*}
 \hat B(\sigma)=V_s\T B(\sigma) V_s=\begin{pmatrix}
  1&1&1&1\\
  \cdot&\cdot&\cdot&\frac1{36\sigma}\\
  \cdot&\cdot&\cdot&\cdot\\
  \cdot&\frac{\sigma}{36}&\frac{\sigma}{18}&\hat b_{44}(\sigma)
 \end{pmatrix},\; \hat b_{44}(\sigma)=\frac{132\sigma+65/\sigma-149}{804}.
\end{align*}
It shows that $\bar B(\sigma)=A^{-1}V_s\mT\hat B(\sigma)V_s^{-1}$ is a smooth function of $\sigma$ near $\sigma=1$.
\par
The matrix $\bar B(\sigma)$ of the forward step of \texttt{AP4o33vgi} is decoupled by a simple weight matrix \eqref{blkivp} with additional scaling in the form
\[ W=\begin{pmatrix}
 1&7& \cdot&\cdot\\[1mm]
 1&\frac{16}5&1&\cdot\\[1mm]
 1&\frac34&\frac34&\frac{14}{25}\\[1mm]
 1&\cdot&\cdot&\cdot
\end{pmatrix}.\]
The norm criterion with no damping ($\nu=0$) can be verified now with a nearly maximal aperture of $\eta_0=17/10\doteq\tan(59.5^o)$ for $\sigma\in{\cal S}=\{\frac9{10},1,\frac{10}9\}$.
On this set, $\hat\zeta_2$ and $\hat\zeta_3$ have positive coefficients for $\ell_2=0$, $\ell_3=2$, while positivity of $\hat\zeta_4$ requires larger exponents $\ell_4=61$ ($\sigma=0.9$), $\ell_4=32$ ($\sigma=1$), and $\ell_4=53$ ($\sigma=10/9$), showing for \texttt{AP4o33vgi} that
\begin{align}\label{BAP4o33}
 \ltnorm(I-z\bar K_n)^{-1}\bar B_n\rtnorm\le1,\ |\Im z|\le-\frac{17}{10}\Re z,
  \ \sigma_n\in\Sc=\Big\{\frac{9}{10},1,\frac{10}{9}\Big\}.
\end{align}
Reducing the aperture to $\eta=14/10\doteq\tan(54.4^o)$, the criterion works with damping $\nu=1/40$ on the same set ${\cal S}$ with the choices
\[\begin{array}{l|c|c|c|}
 k=&2&3&4\\\hline
 \ell_k=&0&2&75
\end{array}\,.
\]
Summarizing, for \texttt{AP4o33vgi} the following bound was proved:
\begin{align}\notag
 \ltnorm(I-z\bar K_n)^{-1}\bar B_n\rtnorm\le\frac1{|1-z/40|},\ |\Im z|\le-\frac{14}{10}\Re z,
  \ \sigma_n\in\Sc.
\end{align}
Considering intervals ${\cal S}=[\underline\sigma,\bar\sigma]$, for $k=4$ the only negative coefficient of $(1+\xi)^{74}\zeta_4(\xi)$ with $\sigma=10/9$ is that of $\xi^{64}$.
Now, with common denominator $(1+\delta)^2$ and $\delta_0=1/9$, the lower bounds \eqref{zetadlt} for the coefficients $\hat\zeta_{4,64}$, $\hat\zeta_{4,65}$ with $\ell_4=75$ are
\begin{align*}
 \hat\zeta_{4,64}\ge&\; 8.2\cdot 10^{12}-6.86\cdot 10^{13}|\delta|,\\
 \hat\zeta_{4,65}\ge&\; 1.93\cdot 10^{12}-1.74\cdot10^{13}|\delta|,
\end{align*}
in floating point form yielding $\hat\delta\doteq\min\{0.11,\delta_0\}$.
\par
The matrix of zero-stability of the adjoint standard step \eqref{APMstd} is
$A\mT B_n\T=(B_nA^{-1})\T=\big(A(A^{-1}B_n)A^{-1}\big)\T$
and obviously possesses the same eigenvalues as that of the forward scheme.
With respect to the full adjoint stability matrix \eqref{Stfna} there seems to be no simple way to transform the estimate \eqref{BAP4o33}, since both of its factors $B(\sigma)A^{-1}$ and $I-zKA_n^{-1}$ depend on parameters and cannot be incorporated in a fixed weight matrix $W^\dagger$. 
Instead we have to repeat the reformulations from Lemma~\ref{LOmdefin} for this case.
Since Pulcherrima is its own adjoint \cite{LangSchmitt2025a}, also its adjoint step has LSRK form, which means that $e_1\T A\mT B\T =e_1\T$.
Hence, the condition of lowest order in \eqref{lokfead} implies $A\mT B\T\eins=\eins$.
With these a simple first transformation of $B(\sigma)A^{-1}$ to block-diagonal form is obtained with $W_a^{\pm1}=I\mp e_1(\eins-e_1)\T$ satisfying $e_1\T W_a=\eins\T$.
Application of the norm criterion to Pulcherrima requires an additional scaling of the last $s-1$ columns, leading to the weight
\begin{align*}
 W^\dagger=\begin{pmatrix}
 1&-\frac{727}{816}&-\frac6{11}&-\frac17\\[1mm]
 \cdot&\frac{29}{16}&\cdot&\cdot\\[1mm]
 \cdot&-\frac43&1&\cdot\\[1mm]
 \cdot&\frac7{17}&-\frac5{11}&\frac17
 \end{pmatrix}.
\end{align*}
With $M=(W^\dagger)^{-1} B(\sigma)A^{-1}W^\dagger$ and $\Gamma=(W^\dagger)^{-1} KA^{-1}W^\dagger$ and appropriate substitutions, we now have with $z=-\xi(1+\imag\eta)$ and fixed $\eta\ge0$ that the norm condition $|1-\nu z|^2\|M(I-z\Gamma)^{-1}x\|_2^2\le\|x\|_2^2$ is equivalent with $y^\ast \Omega^\dagger(\xi)y\ge0$, $y\in\C^s$, $\xi\ge0$, where
\begin{align}\label{Omadef}
\Omega^\dagger(\xi):=
I-M\T M+2\xi\big(\Gamma_S-\imag\eta \Gamma_A-\nu M\T M)+\xi^2(1+\eta^2)(\Gamma\T \Gamma-\nu^2 M\T M),
\end{align}
which differs from $\Omega(\xi)$ by one sign and commuted factors only.
\par
In the convergence analysis below for the boundary value problem \eqref{AWPyp}, \eqref{ADglyp}, only the norm bound without damping will be used.
Hence we will discuss its verification for the case $\nu=0$ only for the sector with $\eta_0=17/10\doteq\tan(59.5^o)$ as before.
Here, we see positivity of all coefficients of the modified principal minors $(1+\xi)^{\ell_k}\zeta_k^\dagger(\xi)$ belonging to \eqref{Omadef} with $\ell_2=1,\,\ell_3=2$ for $\sigma\in{\cal S}=\{0.9,1,10/9\}$ and 
\begin{align*}
 \begin{array}{l|c|c|c|}
  \sigma=&0.9&1&10/9\\\hline
  \ell_4=&44&31&64
 \end{array}\,,
\end{align*}
showing for \eqref{Stfna} that 
\begin{align}\label{BAP4o33a}
\ltnorm \Rc_{n+1}^\dagger(z)\rtnorm\le1,\ |\Im z|\le-\frac{17}{10}\Re z,
  \ \sigma_n\in{\cal S}.
\end{align}
\subsection{Global error estimate}
Estimates for the global error of the boundary value problem may be derived by a straightforward modification of the analysis from \cite{LangSchmitt2025a}.
The essential difference is that now the stiff linear part with the Matrix $J$ will also be treated implicitly and Lipschitz constants of the nonstiff parts $g,\psi$ only are used.
Introducing again numerical grid vectors $Y=(Y_n)_{n=0}^N,\,P=(P_n)_{n=0}^N$ and those for the exact solution as ${\bf y}=({\bf y}_n)_{n=0}^N,\,{\bf p}=({\bf p}_n)_{n=0}^N$, where ${\bf p}_n=\big(p^\star(t_{nj})\big)_{j=1}^s$, $0\le n\le N$, we consider the errors $\check Y_n:=Y_n-{\bf y}_,\,\check P_n:=P_n-{\bf p}_n$, $0\le n\le N$.
In order to avoid doubling of expressions, we also introduce the combined variable $Z=(Y\T,P\T)\T$ and also $G(Z_n):=G(Y_n,P_n),\,\Psi(Z_n):=\Psi(Y_n,P_n)$. 
Subtracting the defining equations of the Peer steps from the same equations, where solution values are used, we arrive at error equations similar to \eqref{FeGl}.
For the adjoint end condition \eqref{APMend}, it reads
\begin{align*}
 (I-h_N\tilde K_NJ\T)\check P_N=&\;\eins\otimes\big(\nabla_y\CC(y_h(T))-\nabla_y\CC(y(T))\big)
 -h_N\tilde K_N\big(\Psi({\bf z}_n+\check Z_n)-\Psi({\bf z}_n)\big)-\tau_N^P,
\end{align*}
with the abbreviation $\tilde K_n:=A_n\mT K$, and $\tilde B_{n+1}:=A_n\mT B_{n+1}\T$ to be used later.
For a polynomial objective function $\CC$ of degree 2, the corresponding term is linear in the error $\check Y_N$ and requires an additional discussion as in \cite{LangSchmitt2025a}.
However, for the sake of shortness, we refrain to linear objective functions, where $\nabla_y\CC$ is a constant vector disappearing in the error equation.
Then the set of error equations becomes
\begin{align}\label{VPEstrt}
 \check Y_0=&\;R_0^Y(\check Z)-(I-h_0\bar K_0J)^{-1}\tau_0^Y,\\\label{VPEstd}
 -\Rc_n(h_nJ)Y_{n-1}+Y_n=&\;R_n^Y(\check Z)-(I-h_n\bar K_nJ)^{-1}\tau_n^Y,\,1\le n\le N,\\\label{APEstd}
  \check P_n-\Rc_{n+1}^\dagger(h_nJ\T)\check P_{n+1}=&\;
  R_n^P(\check Z)-(I-h_n\tilde K_nJ\T)^{-1}\tau_n^P,\ n=0,\ldots,N-1,\\\label{APEend}
 \check P_N
 =&\;R_N^P(\check Z)-(I-h_N\tilde K_NJ\T)^{-1}\tau_N^P,
\end{align} 
with functions $R_n$ defined for $0\le n\le N$ by
\begin{align}\label{DefRny}
 R_n^Y(\check Z):=&\;h_n(\bar K_n^{-1}-h_nI\otimes J)^{-1}\big(G({\bf z}_n+\check Z_n)-G({\bf z}_n)\big),\\\label{DefRnp}
 R_n^P(\check Z):=&\;h_n(\tilde K_n^{-1}-h_nI\otimes J\T)^{-1}\big(\Psi({\bf z}_n+\check Z_n)-\Psi({\bf z}_n)\big).
\end{align}
We note that by using a diagonal matrix $K$ in all time steps, the adjoint error equations are simplified compared to \cite{LangSchmitt2025a}.
We also see here that the redundant formulation of adjoint Peer methods leads to different matrices $\bar K_n=A_n^{-1}K \not=\tilde K_n\T=KA_n^{-1}$, yielding additional degrees of freedom for the triplet.
Ordering terms to the left and right hand sides as in \eqref{VPEstrt}--\eqref{APEend}, the global error equation takes the form
\begin{align}\label{BVPFP}
\MM\check Z
 =-\hat\tau+R(\check Z)
 :=\begin{pmatrix}
 -\hat\tau^Y +R^Y(\check Z)\\
 -\hat\tau^P +R^P(\check Z)
 \end{pmatrix},
 \ \MM:=\begin{pmatrix}
  M_{11}&0\\0&M_{22}
 \end{pmatrix},
\end{align}
where $\hat\tau_n^Y:=(I-h_n\bar K_n)^{-1}\tau_n^Y$, $\hat\tau_n^P:=(I-h_n\tilde K_n)^{-1}\tau_n^P$ and $R(0)=0$.
The matrices $M_{11},M_{22}\in\R^{(N+1)sm\times (N+1)sm}$, are block-bi-diagonal matrices, where the sub-diagonal blocks of size $sm$ in $M_{11}$ are $-\Rc_n(h_nJ)$ from \eqref{VPEstd} and the left-hand side of \eqref{APEstd} shows the super-diagonal blocks $-\Rc_{n+1}^\dagger(h_nJ\T)$ in $M_{22}$ of the same size.
\par
The error for the boundary value problem needs to be analyzed on a global scale with norms on the whole grid for both the state and adjoint variables.
Tight bounds for $\Rc^\dagger(hJ\T)$ are possible in the adjoint norm
\begin{align*}
 \ltnorm \Theta\rtnorm^\dagger:
 =\|((W^\dagger)^{-1}\otimes X)\T\Theta(W^\dagger\otimes X^{-1})\T)\|_2,
 \ \Theta\in\C^{(sm)\times(sm)},
\end{align*}
which is induced by $\ltnorm P_n\rtnorm^\dagger:
 =\|((W^\dagger)^{-1}\otimes X)\T P_n\|_2$. Together with \eqref{WXnorm}, we define for $Z_n=(Y_n\T,P_n\T)\T$ the norm $\ltnorm Z_n\rtnorm:=\max\{\ltnorm Y_n\rtnorm,\ltnorm P_n\rtnorm\}$ and 
for the whole grid vector $Z:=(Z_n)_{n=0}^N$ the global norm
\begin{align*}
 \ltnorm Z\rtnorm:=\max_{n=0}^N\ltnorm Z_n\rtnorm\,.
\end{align*}
Besides the assumption \eqref{XJX} on the stiff part of the ODE, we need again Lipschitz conditions on the nonlinear functions like
\begin{align}\label{LipRWP}
\left.\begin{array}{c}
 \|g(\hat y,\hat p)-g(\tilde y,\tilde p)\|_2\\[1mm]
 \|\psi(\hat y,\hat p)-\psi(\tilde y,\tilde p)\|_2
 \end{array}\right\}
 \le L_{g\psi}\big(\|\hat y-\tilde y\|_2+\|\hat p-\tilde p\|_2\big),
\ \hat y,\tilde y,\hat p,\tilde p\in\R^m.
\end{align}
\begin{theorem}\label{TAdFe}
Consider an adjoint Peer triplet $(A_0,A,A_N,B(\sigma),K)$, where the local errors satisfy \eqref{lokfer}, \eqref{lokfe0}, \eqref{lokfead}, \eqref{lokfean} with $2\le q=r<s$.
Let all real parts of eigenvalues of $K^{-1}A_0,K^{-1}A,K^{-1}A_N$ be bounded from below by $d_{\min}>0$ and assume that there exist constant weight matrices $W,W^\dagger\in\R^{s\times s}$ such that with $M:=W^{-1}\bar B(\sigma)W$, $\Gamma:=W^{-1}\bar KW$ and for $\nu=0$ and some $\eta>0$, the matrix $\Omega(\xi)$ from \eqref{Omdef} is semi-definite for $\xi\ge0,\,\sigma\in{\cal S}$, and in \eqref{Omadef} it holds that $\Omega^\dagger(\xi)\succeq0\;\forall \xi\ge 0,\,\sigma\in{\cal S}$ with $M=(W^\dagger)^{-1}B(\sigma)A^{-1}W^\dagger$, $\Gamma=(W^\dagger)^{-1}KA^{-1}W^\dagger$.
Let the matrix $J=X\Lambda X^{-1}$ in \eqref{AWPyp} with diagonal $\Lambda\in\C^{m\times m}$ have only eigenvalues satisfying $|\Im\lambda(J)|\le-\eta\Re \lambda(J)$, let the Lipschitz condition \eqref{LipRWP} hold for the nonlinear parts $g$ and $\psi$ and $y^\star,p^\star\in C^q[0,T]$ for the solution.
Then, there is a constant $\chi$ depending on the Peer triplet only such that if
\begin{align}\label{ExRWP}
 \chi\,TL_{g\psi}\le\frac12
\end{align}
and $H:=\max_{n=0}^N h_n$ is small enough, a unique solution to the discrete boundary value problem \eqref{VPMstrt}--\eqref{APMend} exists on grids with $\sigma_n\in{\cal S}\subseteq[\underline\sigma,\bar\sigma]$, $1\le n\le N$, and the error satisfies
\begin{align}\label{GlbFeRWP}
 \ltnorm Y_n-{\bf y}_n\rtnorm,\,\ltnorm P_n-{\bf p}_n\rtnorm
 \le C \max_{n=0}^N h_n^{q-1}\max\{\|y^{(q)}\|_{[n]},\|p^{(q)}\|_{[n]}\},
 \,0\le n\le N.
\end{align}
\end{theorem}
{\bf Proof}
Since the shape of the right-hand sides $R_n^Y,R_n^P$ is very similar and simple, its analysis is straightforward.
However, technical difficulties arise since different norms for the state and adjoint variables have to be used.
Considering two vectors $\hat Z_n,\tilde Z_n$ near the origin, with Lemma~\ref{LKIJ} the Lipschitz difference $\ltnorm R_n^Y(\hat Z_n)-R_n^Y(\tilde Z_n)\rtnorm$ of \eqref{DefRny} may be bounded as
\begin{align}\notag
 & \|(W\otimes X)^{-1}\big(R_n^Y(\hat Z_n)-R_n^Y(\tilde Z_n)\big)\|_2\\\notag
 &=h_n\|\big(W^{-1}\bar K_n^{-1}W-h_n\Lambda)^{-1}(W\otimes X)^{-1}
 \big(G({\bf z}_n+\hat Z_n)-G({\bf z}_n+\tilde Z_n)\big)\|_2\\\notag
 &\le h_n\theta\|(W\otimes X)^{-1}\|_2L_{g\psi}(\|\hat Y_n-\tilde Y_n\|_2+\|\hat P_n-\tilde P_n\|_2\big)
\\\label{LipY}
 &\le h_nL^Y\ltnorm \hat Z_n-\tilde Z_n\rtnorm,
\end{align}
with $L^Y= L_{g\psi}\theta\, cond_2(X)\|W^{-1}\|_2(\|W\|_2+\|W^\dagger\|_2)$.
In a similar way, a Lipschitz condition $\ltnorm R_n^P(\hat Z_n)-R_n^P(\tilde Z_n)\rtnorm^\dagger\le h_nL^P\rtnorm \hat Z_n-\tilde Z_n\rtnorm$ may be derived with constant $L^P=L_{g\psi}\theta^\dagger\, cond_2(X)\cdot$ $\|(W^\dagger)^{-1}\|_2\big(\|W\|_2+\|W^\dagger\|_2\big)$, where $\theta^\dagger$ is the constant in Lemma~\ref{LKIJ} derived from $(W^\dagger)^{-1}AK^{-1}W^\dagger$.
In order to obtain a Lipschitz condition for the map $\Phi(Z):=\MM^{-1}(R(Z)-\hat\tau)$, we note that
\begin{align}\notag
 \begin{array}{ll}
 (M_{11}^{-1})_{nk}=\Rc_n(h_nJ)\cdots\Rc_{k+1}(h_{k+1}J),&k<n,\\
 (M_{22}^{-1})_{nk}=\Rc_{n+1}^\dagger(h_{n+1}J\T)\cdots\Rc_k^\dagger(h_kJ\T),&k>n,
 \end{array}
\end{align}
for the individual $m\times m$-blocks of $\MM^{-1}$.
Now, by Theorem~\ref{TNStab}, each of these factors is bounded by one in the appropriate norms with two possible exceptions at the boundaries only.
A stiffness-independent bound for these two exceptions is available through Lemma~\ref{LKIJ} as $\ltnorm\Rc_N(h_NJ)\rtnorm\le\theta\ltnorm\bar K^{-1}\rtnorm$ and $\ltnorm\Rc_1^\dagger(h_0J\T)\rtnorm^\dagger\le\theta^\dagger\ltnorm\bar K\mT\rtnorm^\dagger$ since $\ltnorm\bar B(\sigma_n)\rtnorm=\ltnorm\tilde B(\sigma_1)\T\rtnorm^\dagger=1$.
Hence, for small enough stepsizes $h_0,h_N$, the following bounds hold by Theorem~\ref{TNStab} and \eqref{BAP4o33},
\begin{align}\label{MInv}
 \sum_{k=0}^nh_k\ltnorm(M_{11}^{-1})_{nk}\rtnorm\le 2T,\quad
 \sum_{k=n}^N h_k\ltnorm(M_{22}^{-1})_{nk}\rtnorm^\dagger\le 2T,\ 0\le n\le N.
\end{align}
Combining \eqref{LipY} and the corresponding bound for $R^P$ with the estimate \eqref{MInv} results in the global Lipschitz condition
\begin{align*}
 \ltnorm\Phi(\hat Z)-\Phi(\tilde Z)\rtnorm
 \le L_\Phi\ltnorm\hat Z-\tilde Z\rtnorm,
 \quad L_\Phi:=2T\max\{L^Y,L^P\}.
\end{align*}
With an appropriate factor $\chi$ which may be derived from the explicit representations of $L^Y,L^P$ above, assumption \eqref{ExRWP} means contractivity $L_\Phi\le\frac12$, and images $\Phi(\tilde Z)$ of elements $\tilde Z$ from an $\varepsilon$-ball, $\ltnorm\tilde Z\rtnorm\le\varepsilon$, satisfy
\begin{align*}
 \ltnorm\MM^{-1}\big(-\hat\tau+(R(\tilde Z)-R(0)\big)\rtnorm
 \le \,\ltnorm\MM^{-1}\hat\tau\rtnorm+L_\Phi\ltnorm Z\rtnorm
 \le \,\ltnorm\MM^{-1}\hat\tau\rtnorm+\frac12\varepsilon,
\end{align*}
and are contained in the same ball if $ 2\ltnorm\MM^{-1}\hat\tau\rtnorm\le\varepsilon$.
Hence, there exists a unique solution $\check Z=\Phi(\check Z)$ to \eqref{BVPFP} in this ball, which also satisfies $\ltnorm\check Z\rtnorm\le2\ltnorm\MM^{-1}\hat\tau\rtnorm$.
Since the factors $(I-h_n\bar K_n)^{-1},\,(I-h_n\tilde K_n)^{-1}$ in $\hat\tau$ may again be bounded with Lemma~\ref{LKIJ}, the estimates \eqref{MInv} lead to the error bound
\begin{align}\label{GlbFetau}
 \ltnorm\check Z_n\rtnorm\le C\max_{n=0}^Nh_n^{-1}\max\{\|\tau_n^Y\|_2,\|\tau_n^P\|_2\},\,0\le n\le N,
\end{align}
with some constant $C$.
Finally, the representations \eqref{lokfer}, \eqref{lokfe0} and \eqref{lokfead}, \eqref{lokfean} of the local errors yield the assertion \eqref{GlbFeRWP} for $H$ small enough.
\qed\\
\par\noindent
We note that Theorem~\ref{TAdFe} applies to \texttt{AP4o33vgi} with aperture $\eta=17/10\doteq59.5^o$ according to \eqref{BAP4o33}, \eqref{BAP4o33a}.
Still, it shows convergence of order $q-1$ only.
However, based on the stability estimate \eqref{GlbFetau}, the modification trick from Theorem~\ref{Tsupknv} for proving super-convergence of order $q=s-1$ on smooth grids also for the adjoint equation with appropriate assumptions is a simple repetition of arguments and is avoided here.
The order of convergence of the pulcherrima triplet \texttt{AP4o33vgi} is $q=r=3=s-1$ for both the state and the adjoint variables.
\section{Numerical experiments}\label{NumExp}
In the following, we present comparative results for our novel Peer two-step methods of order $3,4,5$, and well-known L-stable stiffly accurate ESDIRK methods 
(singly diagonally implicit Runge-Kutta methods with a first explicit step) with stage order two. We choose three methods
of order $3,4,5$ from \cite{KennedyCarpenter2016} and one recently developed sixth-order method from
\cite{AlamriKetcheson2024}, see Table~\ref{table:meths}. 
ESDIRK methods are very similar in structure to the Peer two-step methods and differ only in their one-step nature.
We select two typical stiff and singularly perturbed benchmark problems from \cite{AlamriKetcheson2024} 
in order to allow direct comparison with the results presented there, as well as the well-known nonlinear 
semi-discretized Burgers problem, which was already used in 
\cite{Verwer1986} for the numerical investigation of diagonally implicit Runge-Kutta methods. Convergence
orders as well as efficiency results, i.e., computing time versus accuracy, are shown.
\par
All calculations have been done with Matlab-Version R2025b
on a Latitude 7280 with an i5-7300U Intel processor at 2.7 GHz.
\begin{table}
\centering
\centerline{\begin{tabular}{|l|c|c|c|c|l|}\hline
  name & $s$ & $p$ & $\hat{s}$ & $\hat{p}$ & source\\\hline
  \texttt{ESDIRK5o3}& 5 & 3 & 5 & 2 & \cite{KennedyCarpenter2016}, Table~10, p.~81\\
  \texttt{ESDIRK6o4}& 6 & 4 & 6 & 3 & \cite{KennedyCarpenter2016}, Table~16, p.~90\\
  \texttt{ESDIRK7o5}& 7 & 5 & 7 & 4 & \cite{KennedyCarpenter2016}, Table~25, p.~100\\
  \texttt{ESDIRK8o6}& 8 & 6 & 8 & 4 & \cite{AlamriKetcheson2024}, see data availability\\
\hline
\end{tabular}}
\parbox{14cm}{
\caption{List of ESDIRK methods tested: number of stages for the advancing
method $s$, order of convergence of the advancing method $p$, number of stages
for the embedded error estimator $\hat{s}$, and order of convergence for the embedded
error estimator $\hat{p}$. All advancing methods are L-stable and the embedded methods 
for error estimation are A-stable.}
\label{table:meths}
}
\end{table}
\subsection{Stiff Prothero-Robinson problem}
The first stiff benchmark problem was proposed by Prothero and Robinson \cite{ProtheroRobinson1974} 
to test stability and accuracy of stiff ODE integrators. As in \cite{AlamriKetcheson2024},
we consider
\begin{align*}
y'(t) =&\;\mu (y(t)-g(t))+g'(t),\quad y(0)=g(0),\quad t\in [0,1],
\end{align*}
with the stiffness parameter $\mu=-10^3$ and the smooth bounded function
$g(t)=\exp(-t)\cos(20t)+\sin(10t)$. It is well known that implicit one-step
methods including ESDIRK methods suffer from order reduction for stepsizes
larger than $-1/\mu$.
\begin{figure}[t!]
\centering
\includegraphics[width=7cm]{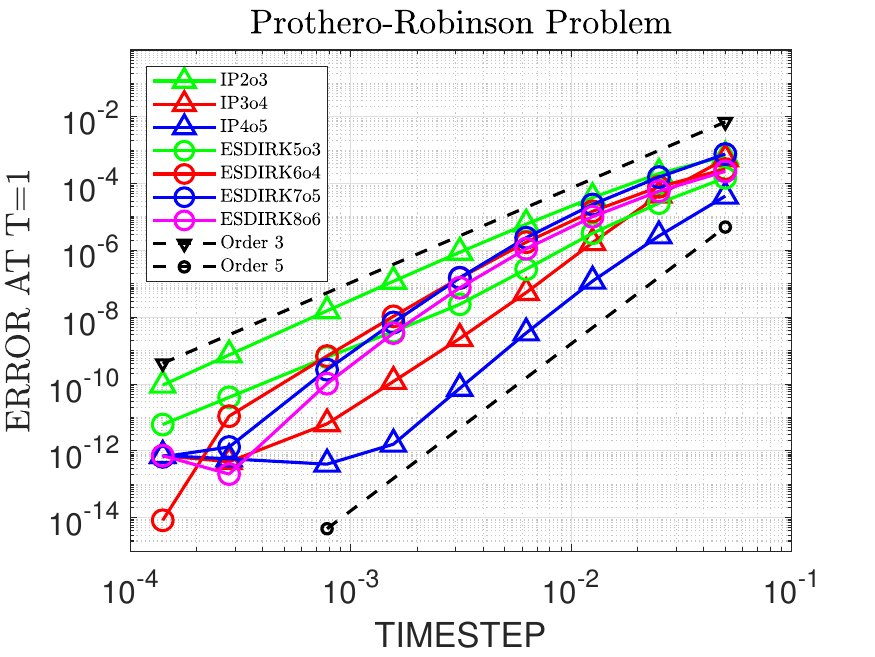}
\hspace{0.1cm}
\includegraphics[width=7cm]{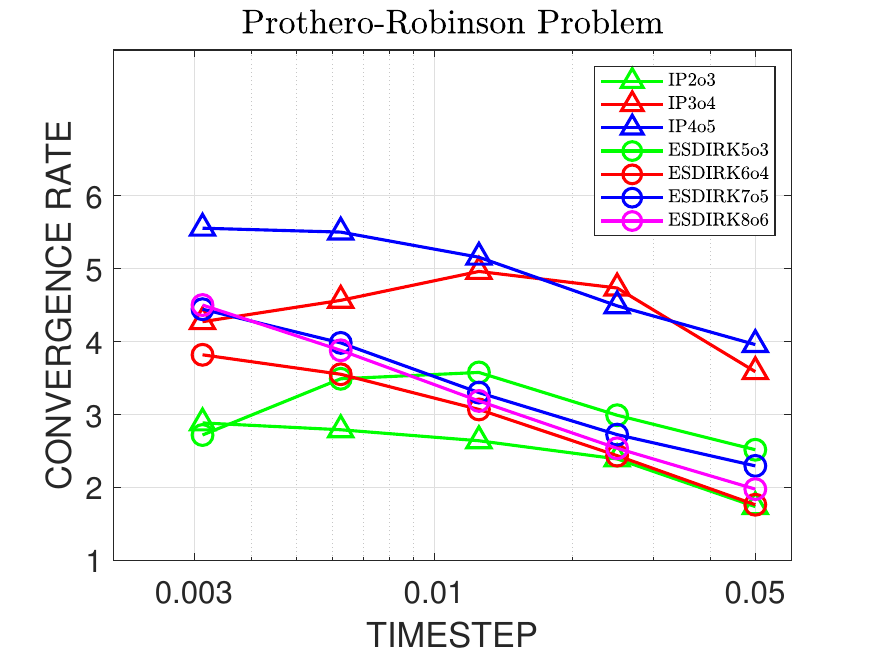}
\parbox{14cm}{
\caption{Prothero-Robinson problem with $\mu=-10^3$: Global errors at 
final time $T=1$ (left) and convergence rates (right) for all methods 
applied with uniform stepsizes.}
\label{fig:protrob}
}
\end{figure}
\par
In Figure~\ref{fig:protrob}, we present global errors at the final time 
$T=1$ for uniform stepsizes $h=2^{-i}/10$, $i=1,\ldots,9$, and convergence rates 
for stepsizes in 
the critical region $h>10^{-3}$. For the Peer methods, the observed numerical
order is close to their design order, whereas a serious order reduction is
visible for all ESDIRK methods, especially for larger time steps. These methods 
behave quite similar in terms of efficiency -- a fact which was already observed 
in \cite{AlamriKetcheson2024} in the case of order reduction. The relatively large
error constant of \texttt{IP2o3} is clearly seen. Both higher-order Peer methods
\texttt{IP3o4} and \texttt{IP4o5} perform considerably better than all ESDIRK
methods before they get affected by round-off errors for $h<10^{-3}$.
\subsection{Singularly perturbed van der Pol system}\label{num:vdp}
The second problem is the famous van der Pol system - a common test example 
that exposes any shortcomings of stiff ODE integrators. Mastering this challenging 
benchmark is essential for designing efficient methods for general singularly perturbed 
problems. We consider
\begin{align}
\label{prob:vanderpol1}
y_1'(t) =&\;y_2(t),\\
\label{prob:vanderpol2}
y_2'(t) =&\;\frac{1}{\varepsilon}\left( (1-y_1(t)^2)y_2(t)-y_1(t)\right),\quad
t\in [0,0.5],
\end{align}
with the initial conditions
\begin{align*}
y_1(0)=&\;2,\quad y_2(t)=-\frac{2}{3}+\frac{10}{81}\varepsilon
-\frac{292}{2187}\varepsilon^2-\frac{1814}{19683}\varepsilon^3+O(\varepsilon^4)\,.
\end{align*}
and $0<\varepsilon\ll 1$, see \cite[Section VI.3]{HairerWanner1996} and \cite{Boscarino2007}
for a detailed analysis and numerical experiments. As $\varepsilon$ approaches zero, the 
stiffness of the system increases. When $\varepsilon=0$, the system degenerates into a semi-explicit 
index-1 DAE. Since no analytic solution is available,
reference solutions are computed with {\sc Matlab}'s \texttt{ode14s} solver setting absolute and relative 
tolerances to $10^{-14}$.
\begin{figure}[t!]
\centering
\includegraphics[width=7cm]{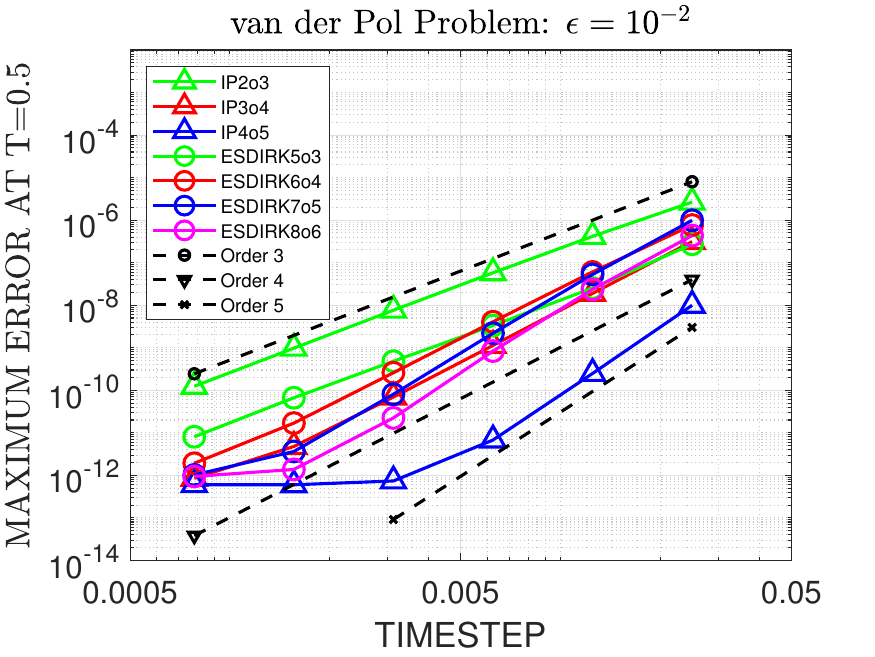}\\
\includegraphics[width=7cm]{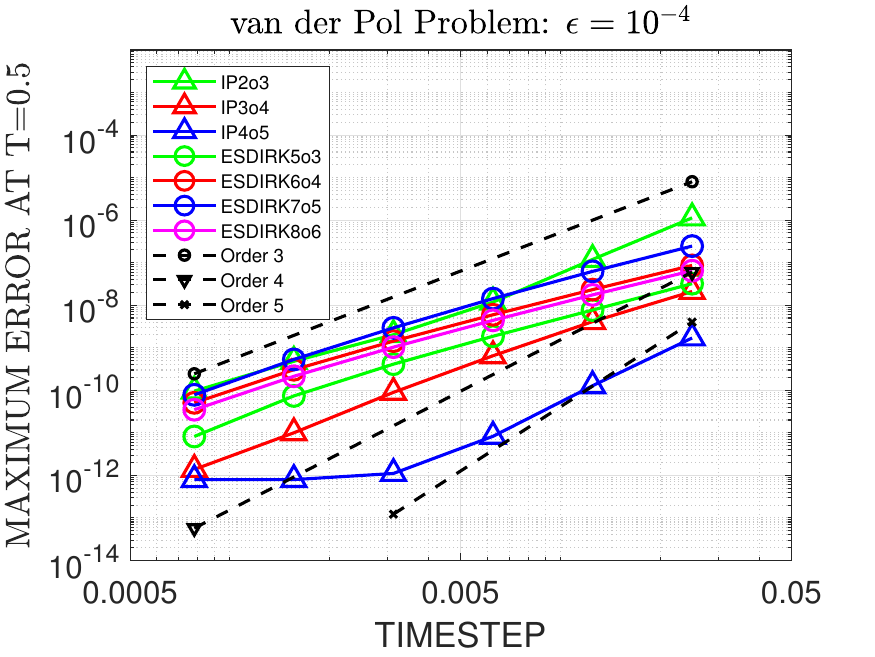}
\hspace{0.1cm}
\includegraphics[width=7cm]{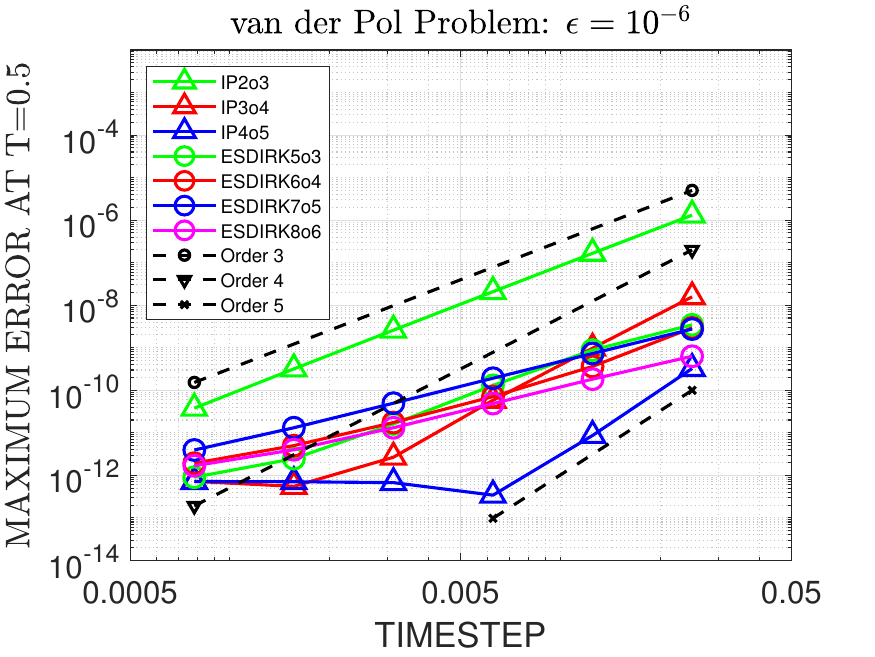}
\parbox{14cm}{
\caption{Van der Pol System: Global errors at final time $T=0.5$ for $\varepsilon=10^{-2},10^{-4},10^{-6}$.}
\label{fig:vanderpol}
}
\end{figure}
\par
Global errors at the final time $T=0.5$ are plotted in Figure~\ref{fig:vanderpol} for stiffness parameters
$\varepsilon=10^{-2},10^{-4},10^{-6}$, employing uniform stepsizes $h=2^{-i}/20$, $i=1,\ldots,6$. For
moderate stiffness $\varepsilon=10^{-2}$, all methods show convergence rates close to their classical order.
In the stiff regime, i.e. $\varepsilon\ll h$, the ESDIRK methods suffer from order reduction and  
convergence rates decrease to their stage order two. This is mostly due to the insufficient resolution of
the emergent algebraic component $y_2$, see \cite[Section VI.3, Corollary 3.10]{HairerWanner1996} and
\cite{AlamriKetcheson2024,Boscarino2007}. The
Peer methods roughly lose one order of convergence for $\varepsilon=10^{-4}$, but regain their full order
at $\varepsilon=10^{-6}$. Clearly, \texttt{IP4o5} performs by far best across all stiffness levels.
\par
Next, we apply all methods in a variable-stepsize implementation, using estimators for the local errors along with a
stepsize selection algorithm. Following again \cite{AlamriKetcheson2024}, we set $\varepsilon=10^{-5}$ and
integrate the system \eqref{prob:vanderpol1}-\eqref{prob:vanderpol2} over $[0,2]$. The error estimators are
defined by
\begin{align}
\label{est:esdirk}
est^{(n)} =&\; y_n-\hat{y}_n,&\text{ for ESDIRK methods},\\
\label{est:peer}
est^{(n)} =&\; C_eh_n\sum_{i=1}^{s}\sigma_n^{s-1}(s-1)!e_s^T(P_sV_s)^{-1}F(Y_{n-1,i}),&\text{ for Peer methods}.
\end{align} 
Here, $y_n$, $\hat{y}_n$ are the advancing and embedded Runge-Kutta approximations, respectively, with order $\hat{p}<p$, and $e_s\T=(0,\ldots,0,1)\in\R^s$.
Replacing $Y_{n-1}$ in \eqref{est:peer} by exact solutions ${\bf y}_{n-1}$, 
Taylor expansion yields $est^{(n)}=C_eh_n^sy^{(s)}(t_n)+O(h_n^{s+1})$, which mimics the leading error term of an embedded solution of order $s-1$. As an averaged error constant, $C_e=10^{-3}$ is used. We work with scaled local error estimators in the weighted $l_2$-metric,
\begin{align*}
w_n :=&\;\sqrt{\frac{1}{m}\sum_{i=1}^{m} \left( \frac{est^{(n)}_i}{rtol\cdot scal\,Y_i^{(n)}+atol}\right)^2},
\end{align*}
where $m=2$ for the van der Pol system, $scal\,Y^{(n)}_i=\max(|y_{n,i}|,|\hat{y}_{n,i}|)$ for ESDIRK methods, 
and $scal\,Y^{(n)}_i=|Y_{n-1,s,i}|$ for Peer methods.
For the Runge-Kutta schemes, a Proportional-Integral controller PI42 is used as in \cite{AlamriKetcheson2024},
resulting in a new stepsize proposal
\begin{align*}
\hat{h}_{n+1} =&\;w_{n}^{-0.6/(\hat{p}+1)}w_{n-1}^{0.2/(\hat{p}+1)}h_n.
\end{align*}
If $w_{n}<1$, the current step is accepted and the computation is continued with an adapted 
time step $h_{n+1}=\min(f_{max}h_n,f_{save}\hat{h}_{n+1})$. We chose $f_{max}=1.5$ and $f_{save}=0.95$ except for 
\texttt{ESDIRK8o6} where the larger
$f_{save}=0.99$ yields much better results. Otherwise the step is rejected and repeated with the classical P-controller, i.e.,
\begin{align}\label{pCtr}
h_{n}^{new}:=&\;\min(f_{max},\max(f_{min},f_{save}w_n^{-1/(\hat{p}+1)}))h_n, 
\end{align}
where we chose $f_{min}=0.6$ and $f_{save}=0.95$ for all methods.
We recall that the embedded Runge-Kutta solutions $\hat{y}_n$ are only A-stable, hence, the error
estimators, too. Consequently, the configuration of the stepsize
controller to be applied in very stiff regimes has to be done with special care.
\par
For Peer methods, we only employ the classical P-controller \eqref{pCtr} with $\hat{p}=s-1$, $f_{save}=0.95$, and $f_{min}=0.8$. 
According to the properties collected in Table~\ref{TPT}, we individually set $f_{max}=\bar\sigma$. If $w_n<1$, we proceed with
a new step $h_{n+1}:=h_n^{new}$, otherwise we repeat the step with $h_n^{new}$.
\par
In Figure~\ref{fig:vanderpol}, we show results of a couple of runs with 
$atol=rtol=10^{-i}$, $i=6,\ldots,10$, and compute weighted root mean-squared errors at $T=2$ from
\begin{align}\label{rms}
err_m :=&\;\sqrt{ \frac{1}{m}\sum_{i=1}^m \left( 
\frac{y_{ref,i}-y_{h,i}}{|y_{ref,i}|+1} \right)^2},\; m=2,
\end{align}
where $y_{ref}$ is the computed reference solution and $y_h$ is the corresponding numerical approximation 
delivered by the methods tested. The initial step is $h_0=10^{-2}$. As measure of efficiency, we show 
accuracy versus computing time in Figure~\ref{fig:vanderpol_cpu}. As already observed in 
\cite{AlamriKetcheson2024}, we can also state that the higher-order ESDIRK methods are less efficient here 
than the lower order ones due to order reduction. The higher-order Peer methods perform superior, especially
\texttt{IP5o4} for sharper tolerances, whereas \texttt{IP2o3} is not competitive in this field of methods.
\begin{figure}[t!]
\centering
\includegraphics[width=8cm]{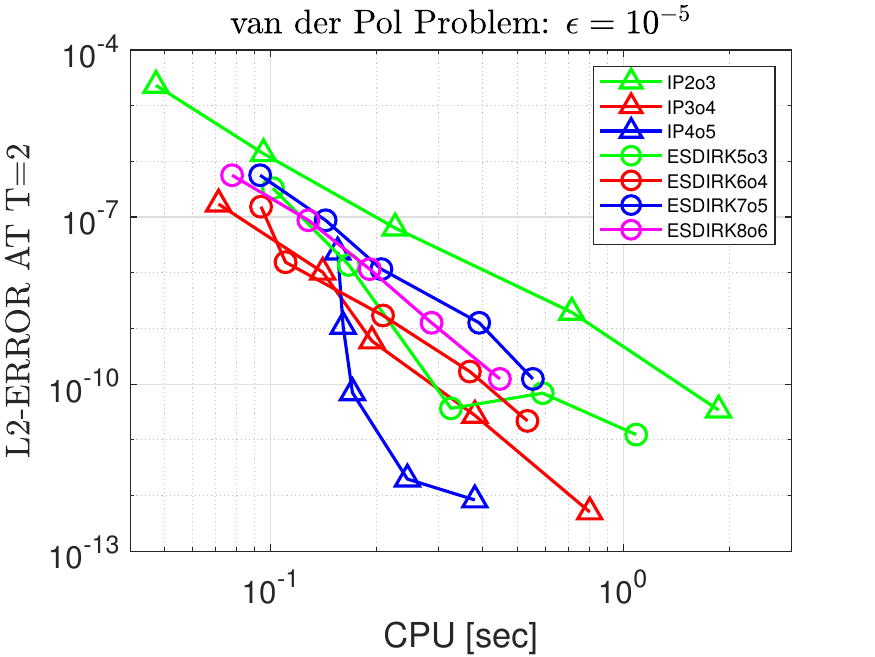}
\parbox{14cm}{
\caption{Van der Pol Problem: Weighted root mean-squared errors $err_2$ defined in \eqref{rms} at the final time $T=2$ 
versus computing time for variable stepsizes with adaptive local error control.}
\label{fig:vanderpol_cpu}
}
\end{figure}
\subsection{Stiff semi-discretized Burgers problem}
As a larger system of stiff ODEs, we consider the semi-discretized Burgers equation which arises 
from a method of lines discretization of the one-dimensional nonlinear convection-diffusion equation 
\begin{align*}
\partial_t u(x,t) =&\; \nu\,\partial_{xx} u - u\partial_xu, \quad (x,t)\in [0,1]\times [0,1].
\end{align*}
We use $\nu=10^{-1}$ and define initial and Dirichlet boundary values from the exact solution given
in \cite[Chapter 4]{Whitham1974},
\begin{align*}
u(x,t) =&\; 1 - 0.9\frac{r_1}{r_1+r_2+r_3}-0.5\frac{r_2}{r_1+r_2+r_3},\\[2mm]
r_1=\exp\left(-\frac{x-0.5}{20\nu}-\frac{99t}{400\nu}\right),\;
r_2=&\;\exp\left(-\frac{x-0.5}{4\nu}-\frac{3t}{16\nu}\right),\;
r_3=\exp\left(-\frac{x-0.375}{2\nu}\right).
\end{align*}
Standard fourth-order central and one-sided finite differences are applied on a uniform grid 
with $M=800$ spatial points. Burgers equation is renowned for its simplicity and ability to 
encapsulate essential phenomena in fluid dynamics and has become a cornerstone in testing
numerical schemes. In Figure \ref{fig:burgers}, we present results from a convergence study 
with uniform time steps $h=2^{-i}/10$, $i=1,\ldots,4$, and from a efficiency test with  
adaptive local error control based on the variable-stepsize implementation described in
Section~\ref{num:vdp}. For uniform stepsizes, all ESDIRK methods show remarkably similar performance and 
demonstrate a clear third-order convergence rate in the maximum norm, exceeding their stage 
order of two by one order. Surprisingly, this is also true for the $l_2$-norm. Hence,
here ESDIRK methods perform better than DIRK methods which reach the fractional order $2.25$ only,
see \cite[Table 5.1]{Verwer1986} and the theoretical discussion in \cite{LubichOstermann1995}. 
No order reduction is observed for the Peer methods, 
whereas \texttt{IP4o5} reaches the level of spatial accuracy $O(10^{-12})$ quite quickly.
\begin{figure}[t!]
\centering
\includegraphics[width=7cm]{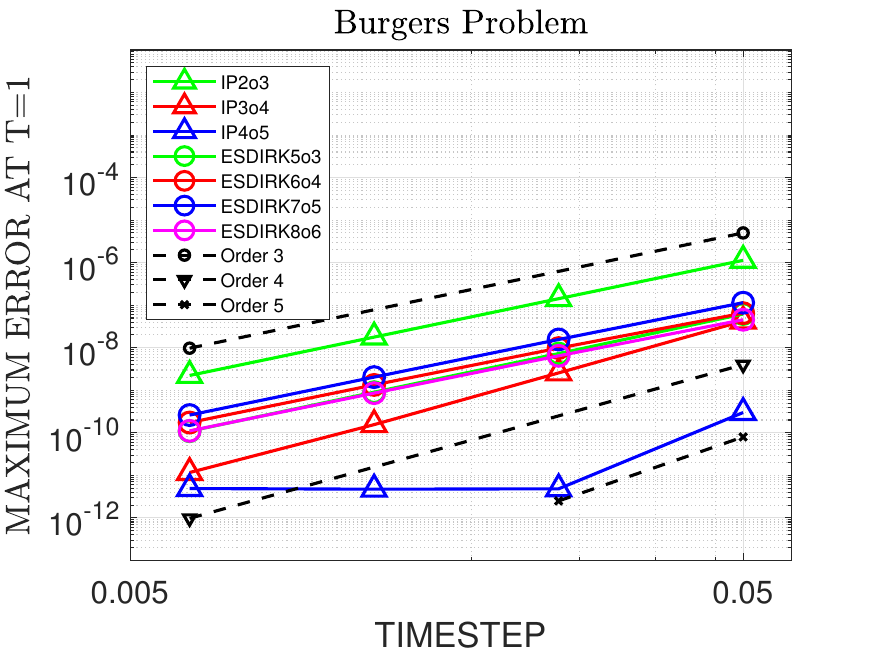}\hspace{0.1cm}
\includegraphics[width=7cm]{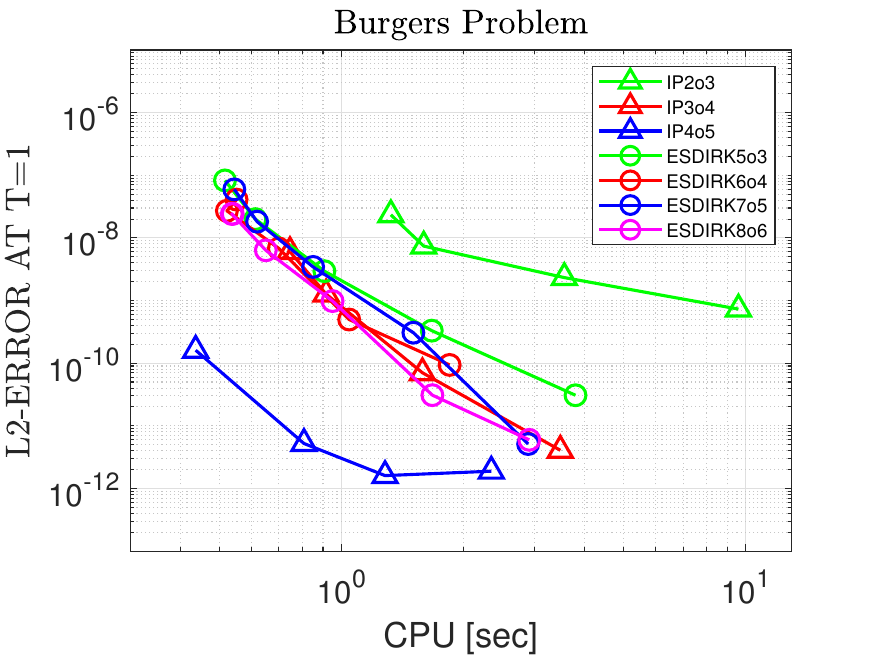}
\parbox{14cm}{
\caption{Burgers Problem with $\nu=10^{-1}$ and $M=800$: Global maximum errors for uniform 
timesteps (left) and weighted root mean-squared errors at final time $T=1$ versus computing 
time for variable stepsizes with adaptive local error control (right).
}
\label{fig:burgers}
}
\end{figure}
\par
Next, we perform a couple of adaptive runs with $atol=rtol=10^{-i}$, $i=8,\ldots,11$ and compute weighted root mean-squared errors at $T=1$ from \eqref{rms} with $m=800$.
The initial step is $h_0=10^{-3}$. From Figure~\ref{fig:burgers} (right), we can draw the following 
conclusions: \texttt{IP2o3} performs less efficiently, whereas \texttt{IP4o5} is by far the most
efficient method, outperforming the others by remarkable three orders of magnitudes. 
The computing times of ESDIRK methods
are very close to each other with the exception of \texttt{ESDIRK3o5} that is less efficient for
sharper tolerances. \texttt{IP3o4} compensates for the greater computational effort required to solve 
the nonlinear stage equations - three LU decompositions instead of one, as in other ESDIRK methods - 
by achieving a higher fourth-order convergence rate.
\section{Conclusion}\label{Concl}
This paper establishes an a priori stability framework for certain variable-stepsize $L(\alpha)$-stable Peer two-step methods applied to stiff semi-linear systems, providing rigorous and uniform stability bounds on variable grids.
The theory applies to a class of Peer methods with a special grid-independent eigenstructure, and we hereby show that $L(\alpha)$-stable methods can achieve order $s+1$ with only $s$ stages, surpassing the traditional stage-order barrier while supporting efficient starting and iteration procedures.
The novel proof technique also applies to
boundary value problems arising in optimal control problems constrained by stiff ODEs and is used to
establish stiff norm estimates for the recently developed extraordinary self-adjoint pulcherrima triplet \texttt{AP4o33vgi}. Numerical experiments on standard stiff benchmark problems confirm the practicality and competitiveness of the proposed novel high-order methods
\texttt{IP3o4} and \texttt{IP4o5}, where especially the latter one clearly outperforms highly respected
and well-established ESDIRK methods.
\vspace{0.5cm}
\par
\noindent {\bf Acknowledgements.}
The first author is supported by the Deutsche Forschungsgemeinschaft
(German Research Foundation) within the collaborative research center
TRR154 {\em ``Mathemati\-cal modeling, simulation and optimisation using
the example of gas networks''} (Project-ID 239904186, TRR154/3-2022, TP B01).
\appendix
\section{Coefficients of the Peer methods}
For the new initial value Peer methods with $s$ stages, denoted by \texttt{IP}$(s)$\texttt{o}$(s+1)$, the node vector  $\cc=(c_1,\ldots,c_s)\T$ and the lower triangular coefficient matrix $K(\sigma)$ will be displayed, the other coefficient matrix $B(\sigma)=\bar B(\sigma)$ may be computed by
\begin{align*}
  B(\sigma)=\big(V_s-K(\sigma)V_s\tilde E_s\big)S(\sigma)\PP_sV_s^{-1},
\end{align*}
where $V_s=(\eins,\cc,\cc^2,\ldots,\cc^{s-1})$ is the Vandermonde matrix, $\tilde E_s=\big(i\delta_{i,j-1}\big),\PP=\big({j-1 \choose i-1}\big)=\exp(\tilde E_s)\in\R^{s\times s}$.
In the step through $[t_n,t_{n+1}]$, $n\ge1$, the coefficients $K_n=K(\sigma_n)$, $B_n=B(\sigma_n)$ depend on the stepsize ratio $\sigma_n=h_n/h_{n-1}$, which is also used in the matrix $S_n=S(\sigma_n)=\diag(1,\sigma_n,\ldots,\sigma_n^{s-1})$.
\par
In addition, the coefficient matrices $A_0$ for the starting step \eqref{Startab} are presented and the diagonals of $\tilde A_0$ to be used in the iterative solution \eqref{Nwtit}.
The sub-diagonals of $A_0$ and $\tilde A_0$ are identical.
\subsection{Coefficients of \texttt{IP2o3}}\label{AIP2o3}
\[ 
\cc=\begin{pmatrix}\frac13\\[1mm]
1\end{pmatrix},
\  K(\sigma)=\begin{pmatrix}
\frac{2+\sigma}{6(1+\sigma)}&0\\[1mm]
\frac34&\frac14
\end{pmatrix},
\ A_0=\begin{pmatrix}6&0\\[1mm]
0&2\end{pmatrix}=\tilde A_0.
\] 
\subsection{Coefficients of \texttt{IP3o4}}\label{AIP3o4}
$\cc\T=\big(\frac19,\frac7{12},1\big)$,
$\diag(\tilde A_0)=\big(16,\frac{17}3,\frac{13}3\big)$,
\[ K(\sigma)=\begin{pmatrix}
 \frac{120+47\sigma+4\sigma^2}{18(60+47\sigma+6\sigma^2)}&0&0\\[1mm]
  \frac{2080+1645\sigma-441\sigma^2}{96(60+47\sigma+6\sigma^2)}&\frac29&0\\[1mm]
  \frac{81}{272}&\frac{48}{85}&\frac{11}{80}
\end{pmatrix},
\ A_0=\begin{pmatrix}
  \frac{270}{17}& \frac{64}{833}& 0\\[1mm]
  -\frac{3969}{68}& \frac{660}{119}& 0\\[1mm]
  -\frac{81}{17}& -\frac{5568}{833}& \frac{13}3
\end{pmatrix}.
\] 
\subsection{Coefficients of \texttt{IP4o5}}\label{AIP4o5}
$\cc\T=\big(\frac16,\frac37,\frac{31}{50},1\big)$,
\[ K(\sigma)\doteq\begin{pmatrix}
  \kappa_{11}&0&0&0\\
  \kappa_{21}&0.1034482758620690&0&0\\
  \kappa_{31}&0.1034482758620690&0.1333333333333333&0\\
  \kappa_{41}&-0.1357417458163727&0.6016742910832833&0.1153508771929825
\end{pmatrix},
\]
\begin{align*}
 q:=&\;1+1.86052631578947\sigma+0.821929824561404\sigma^2+0.10233918128655\sigma^3,\\
q\kappa_{11}=&\;0.166666666666667+0.155043859649123\sigma+0.0456627680311891\sigma^2\\
 &\;+0.00426413255360624\sigma^3,\\
q\kappa_{21}=&\;0.32512315270936+0.530269639616282\sigma+0.21417936114036\sigma^2\\
 &\;+0.00642880953262819\sigma^3,\\
q\kappa_{31}=&\;0.383218390804598+0.727819564428312\sigma+0.271886713848661\sigma^2\\
 &-0.0658600984925368\sigma^3,
\\ \kappa_{41}=&\;0.4187165775401070\,.
\end{align*}
\[ A_0\doteq\begin{pmatrix}
   8.60951871657754& -0.0080439553076369& 0.3529071197439584&  -0.04\\
 -17.18372482682913& 6.667819547538599& 0.02782463527251044& a_{24}^{(0)}\\
  14.41089625668449& -14.79708101914669& 8.94258203382305& 0.12\\
 -11.91176470588235& 25.13484660033167& -23.01826775408627& 6.5625
\end{pmatrix},\]
\begin{align*}
a_{24}^{(0)}=&\;0.09907120743034056,\\
\diag(\tilde A_0)\doteq&\;
(8.691082376542441,6.362899934889681,9.990214081789681,6.610685774659016).
\end{align*}

\bibliographystyle{plain}
\bibliography{bibpeeropt}

\end{document}